\documentclass[dvips,preprint]{imsart}

\RequirePackage[OT1]{fontenc}
\RequirePackage{amsthm,amsmath}
\usepackage{graphicx}
\usepackage{amsthm,amssymb,amsmath,color}
\RequirePackage[numbers]{natbib}
\RequirePackage[colorlinks,citecolor=blue,urlcolor=blue]{hyperref}

\allowdisplaybreaks
\def\E{\mathbb E}
\def\C{\mathbb C}
\def\Z{\mathbb Z}
\def\N{\mathbb N}

\startlocaldefs
\numberwithin{equation}{section}
\theoremstyle{plain}

\newtheorem{dfn}{Definition}
\newtheorem{Prop}{Proposition}
\newtheorem{lemma}{Lemma}

\newtheorem{thm}{Theorem}
\newtheorem{cor}{Corollary}
\newtheorem{rmk}{Remark}
\newtheorem{exm}{Example}

\endlocaldefs

\begin{document}

\begin{frontmatter}
\title{Prediction of time series by statistical learning: general losses and
fast rates}
\runtitle{Prediction of time series}
\thankstext{T1}{We deeply thank Matthieu Cornec (INSEE) for useful discussions,
and
for providing
the data with detailed explanations.
We would like to thank
Prs. Olivier Catoni, Paul Doukhan, Pascal Massart and
Gilles Stoltz for useful comments.
We want to mention
that a preliminary version of
Theorem \ref{main_result}
appeared in proceedings of DS'12 \cite{alquier_li}.
Research partially supported by the ``Agence
Nationale pour la Recherche'', grant ANR-09-BLAN-0128 ``PARCIMONIE''
and grant ANR-Bandhits.}

\begin{aug}
\author{\fnms{Pierre}
\snm{Alquier}$^{(1)}$\ead[label=e1]{pierre.alquier@ucd.ie}}
          \ead[label=u1,url]{http://alquier.ensae.net/},
\author{\fnms{Xiaoyin}
\snm{Li}$^{(2)}$\ead[label=e2]{xiaoyin.li@u-cergy.fr}}
\and
\author{\fnms{Olivier} \snm{Wintenberger}$^{(3,4)}$
\ead[label=e3]{wintenberger@ceremade.dauphine.fr}
\ead[label=u3,url]{http://wintenberger.fr/}}

\runauthor{P. Alquier et al.}

\address{(1) University College Dublin \\
School of Mathematical Sciences \\
Belfield\\
Dublin 4 - Ireland \\
\printead{e1}\\
\printead{u1}}

\address{(2) Universit\'e de Cergy-Pontoise site Saint-Martin \\
  Laboratoire de Math\'ematiques \\
  2, boulevard Adolphe Chauvin\\
  95000 Cergy-Pontoise, France \\
\printead{e2}
}

\address{(3) Universit\'e Paris Dauphine - CEREMADE \\
Place du Mar\'echal de Lattre de Tassigny\\
75775 Paris CEDEX 16, France\\
\printead{e3}\\
\printead{u3}}

\address{(4) CREST-LFA \\
15, boulevard Gabriel P\'eri\\
92245 Malakoff CEDEX, France}
\end{aug}

\begin{abstract}
We establish rates of convergences in time series forecasting using the
statistical
learning approach based on oracle inequalities.
A series of papers (e.g. \cite{modha,meir,baraud:comte:viennet:2001,alqwin})
extends
the oracle inequalities obtained for iid observations to time series under weak
dependence
conditions. Given a
family of predictors and $n$ observations, oracle inequalities state that a
predictor
forecasts the series as well as the best predictor in the family up to a
remainder term
$\Delta_n$. Using the PAC-Bayesian approach, we establish under weak dependence
conditions
oracle inequalities with optimal rates of convergence $\Delta_n$. We extend
results given
in \cite{alqwin} for the absolute loss function to any Lipschitz loss function
with rates
$\Delta_n\sim\sqrt{ c(\Theta)/ n}$ where $c(\Theta)$ measures the complexity of
the model.
We apply the method for quantile loss functions to forecast the french GDP.
Under additional
conditions on the loss functions (satisfied by the quadratic loss function) and
on the time
series, we  refine the rates of convergence to $\Delta_n \sim c(\Theta)/n$.
We achieve for
the first time these fast rates for 
uniformly mixing processes. These rates are known to be optimal in the iid case,
see \cite{TsybakovAgg},
and for individual sequences, see \cite{lugosi}.
In particular, we generalize the results of \cite{DT1}
on sparse regression estimation to the case of autoregression.
\end{abstract}



\end{frontmatter}

\section{Introduction}

Time series forecasting is a fundamental subject  in the mathematical statistics
literature.
The parametric approach contains a wide range of models associated with
efficient
estimation
and prediction methods, see e.g. \cite{hamilton}. Classical parametric
models include linear processes such as ARMA models \cite{davis}. More recently,
non-linear processes
such as stochastic volatility and ARCH models received a lot of attention
in financial applications - see, e.g., the seminal paper by Nobel prize winner
\cite{engle}, and \cite{zak} for a more recent introduction.
However,  parametric assumptions rarely hold on data. Assuming that the data
satisfy a
model can biased the prediction and underevaluate the risks,
see among others
the the polemical but highly informative discussion in \cite{blackswan}.

In the last few years, several universal approaches emerged from various fields
such as non-parametric statistics, machine learning, computer science and game
theory. These approaches share some common features: the aim is to
build a procedure that predicts the time
series as well
as the best predictor in a given set of initial predictors $\Theta$, without any
parametric
assumption on the distribution of the observed time series.
However, the set of predictors can be inspired by different parametric or
non-parametric statistical models. We can distinguish two
classes in these approaches, with different quantification of the objective,
and different terminologies:
\begin{itemize}
\item in the ``prediction of individual sequences'' approach, predictors are
usually
called ``experts''. The objective is online prediction: at each date $t$, a
prediction of the future realization $x_{t+1}$ is based on the previous
observations
$x_1$, ..., $x_t$,
the objective being to minimize the cumulative prediction loss. See for example
\cite{lugosi,stoltz} for an introduction.
\item in the statistical learning approach, the given predictors are sometimes
referred as ``models'' or ``concepts''. The batch setting is more classical in
this approach.
A prediction procedure is built on a complete sample $X_1$, ..., $X_n$. The
performance
of the procedure is compared on the expected loss, called the risk, with the
best predictor, called the ``oracle''.
The environment is not deterministic and some hypotheses like mixing or weak
dependence are
required: see \cite{meir,modha,alqwin}. \end{itemize}

In both settings, one is usually able to predict a time series as well as
the best model or expert, up to an error term that decreases with the number of
observations $n$. This type of results is referred in
statistical theory as oracle inequalities. In other words,
one builds on the basis of the observations a predictor $\hat{\theta}$
 such that
\begin{equation}
\label{objective-eq}
R(\hat{\theta}) \leq \inf_{\theta\in\Theta}R(\theta) + \Delta(n,\Theta)
\end{equation}
where $R(\theta)$ is a measure of the prediction risk of the predictor
$\theta\in\Theta$.
In general, the remainder term
is of the
order $\Delta(n,\Theta)\sim \sqrt{c(\Theta)/n}$ in both approaches, where
$c(\Theta)$ measures the
complexity of $\Theta$. See, e.g., \cite{lugosi} for the
``individual sequences''
approach; for the ``statistical learning approach''
the rate $\sqrt{c(\Theta)/n}$ is reached in \cite{alqwin} with the absolute loss
function
and under a weak dependence assumption.
Different procedures are used to reach these rates. Let us mention the empirical
risk minimization \cite{Vapnik} and aggregation procedures with exponential
weights,
usually referred as EWA \cite{DT1,gerchi} or Gibbs estimator
\cite{Catoni2004,Catoni2007}
in the batch approach, linked to the weighted majority
algorithm of the online approach \cite{LiWa}, see also \cite{VOVK}. Note that
results from the ``individual
sequences'' approach can sometimes be extended to the batch setting, see e.g.
\cite{gerchi}
for the iid case, and \cite{AGARWAL,AGARWAL2} for mixing time series.

In this paper, we extend the results of \cite{alqwin} to the case of a general
loss function. Another improvement with respect to \cite{alqwin} is to study
both
the ERM and the Gibbs estimator under various hypotheses.  We achieve here
inequalities
of the form of~\eqref{objective-eq} that hold with large probability
($1-\varepsilon$ for
any arbitratily small confidence level
$\varepsilon>0$) with $\Delta(n,\Theta)\sim \sqrt{c(\Theta)/n}$. We assume
to do so
that the observations are taken from a bounded
stationary process $(X_t)$ (see \cite{alqwin} however for some possible
extensions to unbounded
observations). We also assume weak dependence conditions on the process
process $(X_t)$.
Then we prove that
the fast rate $\Delta(n,\Theta) \sim c(\Theta)/n$ can be reached for some loss
functions
including the quadratic loss. Note that \cite{meir,modha}
deal with the quadratic loss, their rate can be better than $\sqrt{c(\Theta)/n}$
but
cannot reach $c(\Theta)/n$.

Our main results are based on PAC-Bayesian oracle inequalities.
The PAC-Bayesian point of view
emerged in
statistical learning in supervised classification using the
$0/1$-loss, see the
seminal papers \cite{STW97,McA2}. These results were then extended to general
loss
functions
and more accurate bounds were given, see for example
\cite{Catoni2004,Catoni2007,AlquierPAC,AudibertHDR,AL,Seldin,DalSal}.
In PAC-Bayesian inequalities the complexity term $c(\Theta)$ is defined thanks
to a prior distribution on the set $\Theta$.

The paper is organized as follows:
Section~\ref{section_context} provides notations used in the whole paper.
We give a definition of the Gibbs estimator and of the ERM in
Section~\ref{section_description}.
The main hypotheses necessary to prove theoretical results on these estimators
are provided in Section~\ref{section_hypothesis}.
We give examples of inequalities of the form~\eqref{objective-eq}
for classical set of predictors $\Theta$ in Section~\ref{section_examples}. When
possible,
we also
prove some results on the ERM in these settings. These results only require a
general weak-dependence type assumption on the time series to forecast.
We then study fast rates under a stronger $\phi-$mixing assumptions of
\cite{ib62} in
Section \ref{section_fastrates}.
Note that the $\phi$-mixing setting coincides with the one of
\cite{AGARWAL,AGARWAL2} when $(X_t)$ is stationary.
In particular, we are able to generalize the results of \cite{DT1,gerchi,AL}
on sparse regression estimation to the case of autoregression.
In Section~\ref{section_application} we provide an application to
French GDP forecasting.
A short simulation study is provided in Section~\ref{section_simulation}.
Finally, the proofs of all the theorems are given in
Appendices~\ref{sectionGENERALPACBAYES} and~\ref{sectionproofs}.

\section{Notations}
\label{section_context}

Let
$X_{1},\ldots,X_{n}$ denote the observations at time $t\in\{1,\dots,n\}$
of a time series
$X=\left(X_{t}\right)_{t\in\mathbb{Z}}$ defined on
$(\Omega,\mathcal{A},\mathbb{P})$. We assume that
this series is stationary and take values in $\mathbb{R}^{p}$
equipped with  the Euclidean norm $\|\cdot\|$.
We fix an integer $k$, that might depend on $n$, $k=k(n)$, and assume that 
family of predictors is available:
$\left\{
f_{\theta}:(\mathbb{R}^{p})^{k}\rightarrow\mathbb{R}^p,\theta\in\Theta\right\}$.
For any parameter $\theta$ and any time $t$,
$f_{\theta}\left(X_{t-1},\ldots,X_{t-k}\right)$ is the prediction
of $X_{t}$ returned by the predictor $\theta$ when given $(X_{t-1},\ldots,X_{t-k})$.
For the sake of shortness, we use the notation:
$$ \hat{X}_{t}^{\theta} = f_{\theta}(X_{t-1},\ldots, X_{t-k}).$$
We assume that $\theta\mapsto f_{\theta}$ is a linear function.
Let us fix a loss function $\ell$ that measures a distance between the forecast
and the actual realization of the series. Assumptions on $\ell$ will be given in
Section~\ref{section_hypothesis}.
\begin{dfn}  For any $\theta\in\Theta$ we define the prediction risk as
$$
R\left(\theta\right)=\mathbb{E}\left[\ell\left(\hat{X}_{t}^{\theta},X_{t}
\right)\right]$$
($R(\theta)$ does not depend on $t$ thanks to the stationarity assumption).
\end{dfn}
Using the statistics terminology, note that we may want to include
parametric set of predictors as well as non-parametric ones (i.e. respectively
finite dimensional and infinite dimensional $\Theta$). Let us mention  classical
parametric and non-parametric families of predictors:
\begin{exm}
\label{exm-ARPRED}
Define the set of  linear autoregressive predictors as
$$ f_{\theta}(X_{t-1},\ldots, X_{t-k}) = \theta_0 + \sum_{j=1}^{k} \theta_{j}
X_{t-j}  $$
for $\theta=(\theta_0,\theta_1,\ldots,\theta_k)
\in\Theta \subset \mathbb{R}^{k+1}$.
\end{exm}
In order to deal with  non-parametric settings, we will also use a
model-selection type notation:
$ \Theta = \cup_{j=1}^{M} \Theta_{j} $.
\begin{exm}
\label{ex-nlin}
Consider non-parametric auto-regressive predictors
$$ f_{\theta}(X_{t-1},\ldots, X_{t-k}) = \sum_{i=1}^{j}
\theta_i \varphi_{i}(X_{t-1},\ldots,X_{t-k})  $$
where $\theta=(\theta_1,\ldots,\theta_j)\in\Theta_{j}\subset\mathbb{R}^{j}$ and
$(\varphi_{i})_{i=0}^{\infty}$ is a dictionnary of functions
$ (\mathbb{R}^{p})^{k} \rightarrow\mathbb{R}^{p} $ (e.g. Fourier basis,
wavelets, splines...).
\end{exm}

\section{ERM and Gibbs estimator}
\label{section_description}

\subsection{The estimators}

As the objective is to minimize the risk $R(\cdot)$, we use the empirical
risk $r_n(\cdot)$ as an estimator of $R(\cdot)$.
\begin{dfn}
For any $\theta\in\Theta$,
$ r_{n}(\theta) = \frac{1}{n-k} \sum_{i=k+1}^{n} \ell \left(
\hat{X}_{i}^{\theta},X_{i}\right) .$
\end{dfn}

\begin{dfn}[ERM estimator \cite{Vapnik}]
We define the Empirical Risk Minimizer estimator (ERM) by
$$ \hat{\theta}^{ERM} \in\arg\min_{\theta\in\Theta} r_{n}(\theta). $$
\end{dfn}

Let $\mathcal{T}$ be a $\sigma$-algebra on $\Theta$ and
$\mathcal{M}_{+}^{1}(\Theta)$
denote the set of all probability
measures on $(\Theta,\mathcal{T})$. The Gibbs estimator depends on a fixed
probability measure $\pi\in\mathcal{M}_{+}^{1}(\Theta)$
called the {\it prior} that will be involved when measuring the
complexity of $\Theta$.

\begin{dfn}[Gibbs estimator or EWA]
\label{def_est}
Define the Gibbs estimator with inverse temperature $\lambda>0$ as
$$ \hat{\theta}_{\lambda} = \int_{\Theta} \theta \hat{\rho}_{\lambda}({\rm
d}\theta),\text{ where } \hat{\rho}_{\lambda}({\rm d}\theta) = \frac{e^{-\lambda
r_{n}(\theta)}\pi({\rm d}\theta) }
                                    { \int e^{-\lambda r_{n}(\theta')}\pi({\rm
d}\theta') }. $$
\end{dfn}
The choice of $\pi$ and $\lambda$ in practice is discussed  in
Section~\ref{section_examples}.

\subsection{Overview of the results}

Our results assert that the risk of the  ERM or 
Gibbs estimator  is   close to $\inf_{\theta}R(\theta)$ up to a 
remainder term $\Delta(n,\Theta)$
called the rate of convergence. For the sake of simplicity, let
$\overline{\theta}\in\Theta$ be such that
$$ R(\overline{\theta}) = \inf_{\theta}R(\theta) .$$
If $\overline \theta$ does not exist, it is replaced by an approximative
minimizer $\overline{\theta}_{\alpha}$ satisfying 
$ R(\overline{\theta}_{\alpha}) \leq  \inf_{\theta}R(\theta) + \alpha $ where 
$\alpha$ is negligible w.r.t.
$\Delta(n,\Theta)$ (e.g.
$\alpha<1/n^2$).
We want to prove that
the ERM satisfies, for any $\varepsilon>0$,
\begin{equation}
\label{PAC-ERM}
\mathbb{P}\left( R\left(\hat{\theta}^{ERM}\right)
  \leq R(\overline{\theta}) + \Delta(n,\Theta,\varepsilon) \right)  \geq
1-\varepsilon
\end{equation}
  where $\Delta(n,\Theta,\varepsilon)\rightarrow0$ as $n\rightarrow\infty$.
  We also want to prove that
and that the Gibbs estimator satisfies, for any $\varepsilon>0$,
\begin{equation}
\label{PAC-Gibbs}
\mathbb{P}\left( R\left(\hat{\theta}_{\lambda}\right)
  \leq R(\overline{\theta}) + \Delta(n,\lambda,\pi,\varepsilon) \right)  \geq
1-\varepsilon
\end{equation}
where $\Delta(n,\lambda,\pi,\varepsilon)\rightarrow0$ as $n\rightarrow\infty$
for some $\lambda=\lambda(n)$.
To obtain such results called {\it oracle inequalities}, we require some
assumptions  discussed in
the next section.

\section{Main assumptions}
\label{section_hypothesis}

We prove oracle inequalities under assumptions of two different types. On the
one hand, assumptions
{\bf LipLoss$(K)$} and {\bf Lip$(L)$} hold respectively on the loss function
$\ell$ and
the set of predictors $\Theta$. In some extent, we choose the loss
function and the predictors, so these assumptions can always be satisfied.
Assumption
{\bf Margin$(\mathcal{K})$} also holds on $\ell$.

On the other hand, assumptions {\bf Bound$(\mathcal{B})$}, {\bf WeakDep$(
\mathcal{C})$},
{\bf PhiMix$( \mathcal{C})$} hold on the dependence and boundedness
of the time series. In practice, we cannot know whether these assumptions are
satisfied on data. However, remark that these assumptions are not parametric 
and are satisfied for many classical models, see
\cite{Doukhan1994,Dedecker2007a}.\\

\noindent {\bf Assumption LipLoss$(K)$, $K>0$}: the loss function $\ell$ is
given
by $\ell(x,x')=g(x-x')$ for
some convex $K$-Lipschitz function $g$ such that $g(0)=0$ and $g\geq 0$.

\begin{exm}
A classical example in statistics is given by $\ell(x,x')=\|x-x'\|$, see
\cite{alqwin}.
It satisfies {\bf LipLoss$(K)$} with $K=1$.
In \cite{modha,meir}, the loss function used is the quadratic loss 
$\ell(x,x')=\|x-x'\|^2$. 
It satisfies {\bf LipLoss$(4\mathcal{\mathcal{B}})$} for time series  bounded by
a constant $\mathcal{B}>0$.
\end{exm}

\begin{exm}
\label{exm_quantile}
The class of quantile loss functions introduced in \cite{quantileReg}  is given
by
$$
\ell_{\tau}(x,y)=
\begin{cases}
\tau\left( x-y\right), & \text{if } x-y>0\\
-\left(1-\tau\right)\left( x-y\right), & \text{otherwise}
\end{cases}
$$
where $\tau\in\left(0,1\right)$ and $x$, $y\in\mathbb{R}$.
The risk minimizer of
$t\mapsto \mathbb{E}(\ell_{\tau}(V-t))$ is the
quantile of order $\tau$ of the random variable $V$. Choosing this loss function
one can deal with  rare events and build
confidence intervals, see  \cite{quantileBook,cherno,quantileBiau}. In this
case, {\bf LipLoss$(K)$} is satisfied with $K=\max(\tau,1-\tau)\leq 1$.
\end{exm}

\noindent {\bf Assumption Lip$(L)$, $L>0$}:
for any $\theta\in\Theta$ there are coefficients
$a_{j}\left(\theta\right)$ for $1\leq j \leq k$ such that, for any $x_1$, ...,
$x_k$ and $y_1$, ..., $y_k$,
$$
\left\|f_{\theta}\left(x_1,\ldots,x_k\right)-f_{\theta}\left(y_1,\ldots,
y_k\right)\right\|
     \leq \sum_{j=1}^{k} a_j\left(\theta\right)\left\Vert x_j-y_j\right\Vert,
     $$
with $\sum_{j=1}^{k}a_{j}\left(\theta\right) \leq L$.

\noindent {\bf Assumption Bound$(\mathcal{B})$, $\mathcal{B}>0$}:
we assume that $\|X_0\| \leq \mathcal{B}$ almost surely.\\

Remark that under Assumptions {\bf LipLoss$(K)$}, {\bf Lip$(L)$} and
{\bf Bound$(B)$}, the
empirical risk is a
bounded random
variable. Such a
condition
is required in the approach of individual sequences. We assume it here for
simplicity but it is possible to extend the slow rates
oracles inequalities to unbounded cases see \cite{alqwin}.

Assumption {\bf WeakDep$(\mathcal{C})$} is about the
$\theta_{\infty,n}(1)$-weak dependence coefficients of
\cite{Rio2000a,Dedecker2007a}.
\begin{dfn}
For any $k>0$, define the $\theta_{\infty,k}(1)$-weak dependence coefficients of
a bounded stationary sequence $(X_t)$ by the relation
\begin{multline*}
\theta_{\infty,k}(1) := \\
 \sup_{  f\in\Lambda_{1}^{k}, 0 < j_1 < \cdots < j_k } 
\Bigl\|
\mathbb{E}\left[f(X_{j_1},\dots,X_{j_\ell})|X_{t},t\leq 0\right]
- \mathbb{E}\left[f(X_{j_1},\dots,X_{j_\ell})\right] \Bigr\|_{\infty}
\end{multline*}
where $\Lambda_{1}^{k}$ is the set of $1$-Lipshitz functions of $k$ variables
$$\Lambda_{1}^{k} = \left\{f:(\mathbb{R}^{p})^{k} \rightarrow \mathbb{R}, \quad
                   \frac{|f(u_1,\ldots,u_k) -
f(u'_1,\ldots,u'_k)|}{\sum_{j=1}^{k} \|u_j - u'_j\| } \leq 1 \right\} .$$
\end{dfn}
The sequence $(\theta_{\infty,k}(1))_{k>0}$
is non decreasing with $k$.
The idea is that as soon as $X_k$ behaves``almost independently'' from $X_0$,
$X_{-1}$,
$...$ then $\theta_{\infty,k}(1)-\theta_{\infty,k-1}(1)$ becomes negligible.
Actually, it is known that for many classical models of stationary time series,
the sequence is upper bounded, see \cite{Dedecker2007a} for details.

\noindent {\bf Assumption WeakDep$( \mathcal{C})$, $\mathcal{C}>0$}:
$\theta_{\infty,k}(1)\leq \mathcal{C} $ for any $k>0$.

\begin{exm}
Examples of processes satisfying {\bf WeakDep$(\mathcal{C})$} are provided
in \cite{alqwin,Dedecker2007a}. It includes Bernoulli shifts
 $ X_t = H(\xi_t,\xi_{t-1},\dots) $
where the $\xi_{t}$ are iid, $\|\xi_0\|\leq b$  and $H$ satisfies a Lipschitz
condition:
\begin{equation*}
 \|H(v_1,v_2,...)-H(v'_1,v'_2,...) \|
\le\sum_{j=0}^\infty a_j\|v_j-v'_j\| \text{ with }
 \sum_{j=0}^\infty j a_j < \infty.
\end{equation*}
Then $(X_t)$ is bounded by $\mathcal{B} = H(0,0,...)
 + b \mathcal{C}$ and satisfies {\bf WeakDep$(\mathcal{C})$}  with $ \mathcal{C}
= \sum_{j=0}^\infty j a_j $.
In particular, solutions of linear ${\rm ARMA}$ models with bounded innovations
satisfy
{\bf WeakDep$(\mathcal{C})$}.
\end{exm}

In order to prove the fast rates oracle inequalities, a more restrictive
dependence
condition  is assumed. It holds on the uniform  mixing coefficients introduced
by \cite{ib62}.
\begin{dfn}
The $\phi$-mixing coefficients
of the stationary sequence $(X_{t})$ with distribution $\mathbb{P}$ are defined
as
$$
\phi_r=\sup_{(A,B)\in\,\sigma(X_t, t\le 0)\times\sigma(X_t,t\ge
r)}|\mathbb{P}(B/A)-\mathbb{P}(B)|.
$$
\end{dfn}

\noindent {\bf Assumption PhiMix$(\mathcal{C'})$, $\mathcal{C'}>0$}: 
$ 1+\sum_{r=1}^{\infty} \sqrt{\phi_{r}}
\leq \mathcal{C'}  .$\\

This assumption appears to be more restrictive than {\bf  WeakDep$(\mathcal{C})$} for
bounded time series:
\begin{Prop}[\cite{Rio2000a}]
$$
\text{{\bf Bound}}(\mathcal{B}) \text{ and {\bf PhiMix}}(\mathcal{C})
 \Rightarrow \text{{\bf Bound}}(\mathcal{B}) \text{ and {\bf  WeakDep}} (\mathcal{C}\mathcal{B}).
 $$
\end{Prop}

(This result is not stated in \cite{Rio2000a} but it is a direct consequence
of the last inequality in the proof of Corollaire 1, p. 907 in \cite{Rio2000a}).

Finally, for fast rates oracle inequalities, an additional assumption
on the loss function $\ell$ is required.
In
the iid case, such a condition is also required. It is called Margin
assumption, e.g. in
\cite{MamTsy,AlquierPAC}, or Bernstein hypothesis, \cite{lecue}.\\

\noindent {\bf Assumption Margin$(\mathcal{K})$, $\mathcal{K}>0$}:
\begin{multline*}
\E
\left\{ \left[\ell\Bigl(X_{q+1},f_{\theta}(X_{q},...,X_{1})\Bigr)
         -
        \ell\Bigl(X_{q+1},f_{\overline{\theta}}(X_{q},...,X_{1})\Bigr) \right]^2
\right\}
 \\ \leq \mathcal{K} \left[R(\theta)-R(\overline{\theta})\right].
 \end{multline*}
 
As assumptions {\bf Margin$(\mathcal{K})$} and {\bf PhiMix$(\mathcal{C})$}
are used only to obtain fast rates, we give postpone examples to Section
\ref{section_fastrates}.

\section{Slow rates oracle inequalities}
\label{section_examples}

In this section, we give oracle inequalities \eqref{PAC-ERM}
and/or~\eqref{PAC-Gibbs} with slow
rates of convergence $\Delta(n,\Theta)\sim
\sqrt{c(\Theta)/n}$. The proof of these results  are given in
Section~\ref{sectionproofs}. 
Note that the results concerning the Gibbs estimator are actually corollaries
of a general result, Theorem~\ref{main_result},
stated in
Section~\ref{sectionGENERALPACBAYES}. We introduce the following
notation for the sake of shortness.

\begin{dfn}
 When Assumptions \textcolor{blue}{{\bf Bound}$(\mathcal{B})$}, {\bf
LipLoss$(K)$}, {\bf Lip}($L$) 
and {\bf WeakDep$( \mathcal{C})$} are satisfied, we
say
that we are under the set of Assumption {\bf
SlowRates($\kappa$)} where $\kappa  = K(1+L) (\mathcal{B} +
\mathcal{C})/\sqrt{2}$ .
\end{dfn}

\subsection{Finite classes of predictors}

Consider first the toy example where $\Theta$ is finite with
$|\Theta|=M$, $M\ge 1$. In this case, the optimal rate in the iid case
is known to be $\sqrt{\log(M)/n}$, see e.g. \cite{Vapnik}.
\begin{thm}
\label{corofinite} Assume that $|\Theta|=M$ and that {\bf SlowRates($\kappa$)}
is 
satisfied for $\kappa>0$. Let $\pi$ be  the uniform probability distribution on
$\Theta$. Then the oracle inequality \eqref{PAC-Gibbs}
is satisfied 
for any $\lambda>0$, $\varepsilon>0$  with
$$
\Delta(n,\lambda,\pi,\varepsilon) =  \frac{2\lambda
 \kappa^{2}}{n\left(1- {k}/{n}\right)^2} +
\frac{ 2 \log\left({2M}/{\varepsilon}\right)
}{\lambda}.
$$
\end{thm}
The choice of $\lambda$ in practice in this toy example is already not trivial.
The choice $\lambda=
\sqrt{\log(M)n}$ yields the oracle inequality:
$$
R (\hat{\theta}_{\lambda} )  \leq
R(\overline{\theta})
+ 2\sqrt{\frac{\log(M)}{n}} \left(\frac{\kappa}{1- {k}/{n}}\right)^2
     + \frac{2\log\left({2}/{\varepsilon}\right)}{\sqrt{n\log(M)}}.
$$
However, this choice is not optimal and one would like to choose $\lambda$ as
the minimizer of the upper bound
$$ \frac{2\lambda \kappa^{2}}{n\left(1-{k}/{n}\right)^2} +
\frac{ 2 \log\left(M\right).
}{\lambda} $$
However $\kappa=\kappa(K,L,\mathcal{B},\mathcal{C})$ and
the constants $\mathcal{B}$ and $\mathcal{C}$ are, usually, unknown. In this
context we will prefer the ERM predictor that performs as well as the
Gibbs estimator with optimal $\lambda$:
\begin{thm}
 \label{thfinite}
Assume that $|\Theta|=M$ and that 
{\bf SlowRates($\kappa$)} is 
satisfied for $\kappa>0$. Then the oracle inequality \eqref{PAC-ERM} is
satisfied for any $\varepsilon>0$ with
$$
\Delta(n,\Theta,\varepsilon) =
\inf_{\lambda>0}
  \left[\frac{2 \lambda \kappa^2 }{n
        \left(1- {k}/{n}\right)^{2}}
     + \frac{
2 \log\left( {2M}/{\varepsilon}\right)}{\lambda} \right] =
\frac{4\kappa}{1-{k}/{n}} \sqrt{\frac{\log\left( {2M}/{\varepsilon}\right)}{n}}
  .
$$
\end{thm}

\subsection{Linear autoregressive predictors}

We focus on the linear predictors given in Example~\ref{exm-ARPRED}. 
\begin{thm} \label{AR-thm1}
Consider the linear autoregressive model of ${\rm AR}(k)$ predictors $$
f_{\theta}(x_{t-1},\ldots, x_{t-k}) = \theta_0 + \sum_{j=1}^{k} \theta_{j}
x_{t-j}  $$
with $\theta \in  \Theta= \{\theta \in \mathbb{R}^{k+1}, \| \theta \|
     \leq L \}$ such that {\bf Lip$(L)$} is satisfied.
Assume that Assumptions {\bf Bound$(\mathcal{B})$}, {\bf LipLoss$(K)$} and
{\bf WeakDep$(\mathcal{C})$} are satisfied.
Let $\pi$ be the uniform probability
distribution on the extended parameter set
$\{\theta \in \mathbb{R}^{k+1}, \| \theta \|
     \leq L+1 \}$. Then the oracle inequality \eqref{PAC-Gibbs}
is satisfied for any $\lambda>0$, 
$\varepsilon>0$ with
\begin{multline*}
\Delta(n,\lambda,\pi,\varepsilon) = \\
\frac{2\lambda \kappa^{2}}{n\left(1-{k}/{n}\right)^2} 
+
2 \frac{ (k+1) \log \left(\frac{(K\mathcal{B}\vee K^2 \mathcal{B}^{2})(L+1)
\sqrt{e} \lambda}{k+1}\right)
  +   \log\left( {2}/{\varepsilon}\right)
}{\lambda}.
\end{multline*}
\end{thm}
In theory, $\lambda$ can be chosen of the order  $\sqrt{(k+1)n} $ to achieve the
optimal rates $\sqrt{(k+1)/n}$ up to a logarithmic factor. But the choice of the
optimal $\lambda$ in practice is still a problem. The ERM predictor still
performs as well as the Gibbs predictor with optimal $\lambda$.
\begin{thm} \label{AR-thm2}
Under the assumptions of Theorem \ref{AR-thm1}, the oracle inequality
\eqref{PAC-ERM}
is satisfied for any
$\varepsilon>0$  with
\begin{multline*}
\Delta(n,\Theta,\varepsilon) = \\
      \inf_{\lambda\geq {2K\mathcal{B}}/({k+1})}
      \left[\frac{2 \lambda \kappa^2 }{n \left(1-{k}/{n}\right)^{2}}
     + \frac{ (k+1)\log\left(\frac{2eK\mathcal{B}(L+1)\lambda}{k+1}\right) +
2 \log\left({2}/{\varepsilon}\right)}{\lambda} \right].
\end{multline*}
\end{thm}

The additional constraint on $\lambda$ does not depend on $n$. It is restrictive
only when $k+1$, the complexity of the autoregressive model, has the same order
than $n$.
For $n$ sufficiently large and
$\lambda=((1-k/n)/\kappa)\sqrt{((k+1)n/2)}$  
satisfying the constraint $\lambda \geq 2K\mathcal{B}/(k+1)$
we obtain the oracle inequality
\begin{multline*}
 R(\hat{\theta}^{ERM})
 \leq R(\overline{\theta})
  \\
      + \sqrt{\frac{2(k+1)}{n}} \frac{\kappa}{1-{k}/{n}}
        \log\left( \frac{2 e^2 K \mathcal{B} (R+1)}{\kappa}\sqrt{\frac{n}{k+1}}
\right)
       \\  +   \frac{2\sqrt{2}\kappa
\log\left({2}/{\varepsilon}\right)}{\sqrt{(k+1)n}
         \left(1-{k}/{n}\right)}.
\end{multline*}
Theorems~\ref{AR-thm1} and \ref{AR-thm2} are both direct consequences of
the following results about general classes of predictors.

\subsection{General parametric classes of predictors}

We state a general result about finite-dimensional families of predictors. The
complexity $k+1$ of the autoregressive model is replaced by a more general
measure of the dimension $d(\Theta,\pi)$. We also introduce some general measure
$D(\Theta,\pi)$ of the diameter that will, for most compact models, be linked
to the diameter of the model.
\begin{thm}
\label{thmGibbs1}
Assume that {\bf SlowRates($\kappa$)} is satisfied and the existence of $
d=d(\Theta,\pi) >0 $ and $D=D(\Theta,\pi)>0$ satisfying the relation
$$ \forall \delta>0,\quad
\log \frac{1} {\int_{\theta\in\Theta}  \mathbf{1} \{ R(\theta) -
R(\overline{\theta})
    < \delta \}  \pi({\rm d}\theta)} 
 \leq d\log\left(\frac{D}{\delta}\right). $$
Then the oracle inequality \eqref{PAC-Gibbs}
is satisfied for any $\lambda>0$,  $\varepsilon>0$ with
$$
\Delta(n,\lambda,\pi,\varepsilon) = \frac{2\lambda
\kappa^{2}}{n\left(1-{k}/{n}\right)^2} 
+
2 \frac{ d\log \left( {D \sqrt{e} \lambda}/{d}\right)
  +   \log\left({2}/{\varepsilon}\right)
}{\lambda}.
$$
\end{thm}
A similar result holds for the ERM predictor under a more restrictive assumption
on the structure of $\Theta$, see Remark~\ref{rmkassumptions} below.
\begin{thm}
\label{thmERM}
Assume that
\begin{enumerate}
 \item $\Theta = \{\theta\in\mathbb{R}^{d}: \|\theta\|_1 \leq D \}$,
 \item 
$ \|\hat{X}_{1}^{\theta_1} - \hat{X}_{1}^{\theta_2}\|
             \leq \psi. \left\| \theta_1 - \theta_2 \right\|_{1}$ a.s. for some
$\psi>0$ and all $(\theta_1,\theta_2)\in\Theta^2$.
\end{enumerate}
Assume also that {\bf Bound$(\mathcal{B})$}, {\bf LipLoss$(K)$} and {\bf WeakDep$(\mathcal{C})$}
are satisfied and that  {\bf Lip}($L$) holds on the extended model
$\Theta' = \{\theta\in\mathbb{R}^{d}: \|\theta\|_1 \leq D+1 \}$.
Then the oracle inequality 
\eqref{PAC-ERM}
is satisfied for any $\varepsilon>0$ with
$$
\Delta(n,\Theta,\varepsilon) = \inf_{\lambda\geq {2K\psi}/{d}}
      \left[\frac{2 \lambda \kappa^2 }{n \left(1-{k}/{n}\right)^{2}}
     + \frac{ d\log\left({2eK\psi(D+1)\lambda}/{d}\right) +
2 \log\left({2}/{\varepsilon}\right)}{\lambda} \right].
$$
\end{thm}
This result yields to nearly optimal rates of convergence for the ERM
predictors. Indeed, 
for $n$ sufficiently large and $\lambda=((1-k/n)/\kappa)\sqrt{(dn/2)}\geq
2K\psi/d$
 we obtain the oracle inequality
$$
 R(\hat{\theta}^{ERM})
 \leq R(\overline{\theta})
      + \sqrt{\frac{2d}{n}} \frac{\kappa}{1- {k}/{n}}
        \log\left( \frac{2 e^2 K \psi (D+1)}{\kappa}\sqrt{\frac{n}{d}} \right)
         +   \frac{2\sqrt{2}\kappa \log\left( {2}/{\varepsilon}\right)}
             {\sqrt{dn} \left(1- {k}/{n}\right)}.
$$
Thus, the ERM procedure yields prediction that are close to the oracle with an
optimal rate of convergence up to a logarithmic factor.

\begin{exm}
Consider the linear autoregressive model of ${\rm AR}(k)$ predictors studied in
Theorems~\ref{AR-thm1}
and~\ref{AR-thm2}. Then {\bf Lip}($L$) is automatically satisfied with
$L=D+1$. The assumptions of
Theorem~\ref{thmERM} are satisfied with
$d=k+1$ and
$\psi=\mathcal{B}$. Moreover, thanks to Remark~\ref{rmkassumptions}, the
assumptions of
Theorem~\ref{thmGibbs1} are satisfied with
$D(\Theta,\pi)=(K\mathcal{B}\vee K^2\mathcal{B}^2)(R+1)$. Then
Theorems~\ref{AR-thm1} and~\ref{AR-thm2} are actually direct consequences of
Theorems~\ref{thmGibbs1}
and~\ref{thmERM}.
\end{exm}

Note that the context of
Theorem~\ref{thmERM} are less general than the one
of Theorem~\ref{thmGibbs1}:
\begin{rmk}
\label{rmkassumptions}
Under the assumptions of
Theorem~\ref{thmERM} we have for any $\theta\in\Theta$
\begin{align*}
R(\theta) - R(\overline{\theta})
    & = \mathbb{E}\biggl\{ g\left(\hat{X}_{1}^{\theta}
            -X_1 \right) - g\left(\hat{X}_{1}^{\overline{\theta}}
                                -X_1 \right) \biggr\} \\
   & \leq \mathbb{E}\biggl\{ K \left\|\hat{X}_{1}^{\theta}
           - \hat{X}_{1}^{\overline{\theta}}\right\|  \biggr\} \\
   & \leq K \psi \| \theta - \overline{\theta} \|_{1}.
\end{align*}
Define $\pi$ as the uniform distribution
on $\Theta'=\{\theta\in\mathbb{R}^{d}: \|\theta\|_1 \leq D+1 \}$. We derive from
simple computation the inequality
\begin{multline*}
\log \frac{1} {\int_{\theta\in\Theta}  \mathbf{1} \{ R(\theta) -
R(\overline{\theta})
    < \delta \}  \pi({\rm d}\theta)}
\leq
\log \frac{1} {\int_{\theta\in\Theta}  \mathbf{1} \{
\|\theta-\overline{\theta}\|_{1}
    < \frac{\delta}{K\psi} \}  \pi({\rm d}\theta)}
\\
\left\{
\begin{array}{l}
= d \log\left(\frac{K\psi(D+1)}{\delta}\right) \text{ when } \delta/K\psi \leq 1
\\ \\
\leq d \log\left(K\psi (D+1)\right) \text{ otherwise}.
\end{array}
\right.
\end{multline*}
Thus, in any case,
$$ \log \frac{1} {\int_{\theta\in\Theta}  \mathbf{1} \{ R(\theta) -
R(\overline{\theta})
    < \delta \}  \pi({\rm d}\theta)}
       \leq d \log \left(\frac{(K\psi\vee K^2 \psi^2)(D+1)}{\delta} \right) $$
and the assumptions of Theorem~\ref{thmGibbs1} are satisfied for
$d(\Theta,\pi)=d$ and  $D(\Theta,\pi)=(K\psi\vee K^2 \psi^2)(D+1)$.
\end{rmk}
As a conclusion, for some predictors set with a non classical structure, the
Gibbs estimator might be preferred to the ERM.

\subsection{Aggregation in the model-selection setting}

Consider now several models of predictors $\Theta_1$, ..., $\Theta_M$ and
consider $ \Theta = \bigsqcup_{i=1}^{M} \Theta_i $
(disjoint union). Our aim is
to predict as well as the best predictors among all $\Theta_j$'s, but paying
only the price for learning in the
$\Theta_j$ that contains the oracle. In order to get such a result,  let
us choose $M$ priors $\pi_j$  on each models such that
$\pi_j(\Theta_j)=1$ for all  $j\in\{1,...,M\}$.  Let
$\pi = \sum_{j=1}^{M} p_j \pi_j $ be a mixture of these priors with prior
weights $p_j\geq 0$ satisfying $\sum_{j=1}^{M} p_j = 1$. Denote 
$$ \overline{\theta}_j \in\arg\min_{\theta\in\Theta_j}R(\theta) $$
the oracle of the model $\Theta_j$ for any $1\le j\le M$. For any $\lambda>0$,
denote $\hat{\rho}_{\lambda,j}$ the Gibbs distribution on
$\Theta_j$ and $\hat{\theta}_{\lambda,j} = \int_{\Theta_j}
\theta\hat{\rho}_{\lambda,j}
({\rm d}\theta)$ the corresponding Gibbs estimator. A Gibbs predictor based on
a model selection procedure satisfies an oracle inequality with
slow rate of
convergence:
\begin{thm}
\label{thmGibbs2}
Assume that:
\begin{enumerate}
 \item {\bf Bound}$(\mathcal{B})$ is satisfied for some
$\mathcal{B}>0$;
\item  {\bf LipLoss$(K)$} is satisfied for some $K>0$;
\item  {\bf  WeakDep$( \mathcal{C})$} is satisfied for some $\mathcal{C}>0$;
\item for any $j\in\{1,...,M\}$ we have
\begin{enumerate}
 \item {\bf Lip}($L_j$) is satisfied
by  the model $\Theta_j$ for some $L_j>0$,
 \item there are constants
 $ d_j=d(\Theta_j,\pi)  $ and $D_j=c(\Theta_j,\pi_j)$ are such that
$$ \forall \delta>0,\quad
\log \frac{1} {\int_{\theta\in\Theta_j}  \mathbf{1} \{ R(\theta) -
R(\overline{\theta}_j)
    < \delta \}  \pi_j({\rm d}\theta)} 
 \leq d_j\log\left(\frac{D_j}{\delta}\right) $$
\end{enumerate}
\end{enumerate}
Denote $\kappa_j=\kappa(K,L_j,\mathcal{B},\mathcal{C})
             = K(1+L_j) (\mathcal{B} + \mathcal{C})/\sqrt{2}$ and define
$ \hat{\theta} = \hat{\theta}_{\lambda_{\hat{j}},\hat{j}} $
where $ \hat{j}$ minimizes the function of $j$
$$ 
             \int_{\Theta_j} r_n(\theta) \hat{\rho}_{\lambda_j,j}
({\rm d}\theta) + \frac{\lambda_j \kappa_j}{n(1-{k}/{n})^2}
          + \frac{\mathcal{K}(\hat{\rho}_{\lambda_j,j},\pi_j) + \log\left( {2}/
                  {(\varepsilon p_j)}\right) }{\lambda_j}
           $$
with
$$ \lambda_j = \arg\min_{\lambda>0}\left[ \frac{2 \lambda \kappa_j^2 }
                       {n \left(1-{k}/{n}\right)^{2}}
            + 2 \frac{ d_j \log\left({D_j e \lambda}/{d_j}\right) +
\log\left({2}/({\varepsilon p_j})\right)}{\lambda} \right] .$$
Then,  with probability at least $1-\varepsilon$, the following oracle
inequality holds
$$
R(\hat{\theta}) \leq
 \inf_{1\leq j \leq M}
    \left[ R(\overline{\theta}_j)
           + 2\frac{\kappa_j}{1- {k}/{n}} \left\{ \sqrt{\frac{d_j}{n}}
              \log\left(\frac{D_j e^2 }{\kappa_j} \sqrt{\frac{n}{d_j}} \right)
                     + \frac{\log\left( {2}/({\varepsilon p_j})\right)}{\sqrt{ n
d_j}}
                \right\}  \right].
$$
\end{thm}

The proof is given in Appendix~\ref{sectionproofs}.
A similar result can  be obtained if we replace the Gibbs predictor
in each  model by the ERM predictor in each  model. 
The resulting procedure is known in the iid case under the name SRM
(Structural Risk Minimization), see \cite{Vapnik}, or penalized risk
minimization, \cite{BirMas}. However, as it was already the case for a fixed
model,
additional assumptions are required to deal with ERM predictors. In the
model-selection 
context, the procedure to choose among all the ERM predictors also depends on
the unknown
$\kappa_j$'s. Thus the model-selection procedure based on Gibbs predictors
outperforms
the one based on the ERM predictors.

\section{Fast rates oracle inequalities}
\label{section_fastrates}

\subsection{Discussion on the assumptions}

In this section, we study conditions under which the rate  $1/n$ can be
achieved. These conditions are restrictive:
\begin{itemize}
\item now $p=1$, i.e. the process $(X_t)_{t\in\mathbb{Z}}$ is real-valued;
\item the dependence condition {\bf  WeakDep$( \mathcal{C})$} is replaced by
{\bf PhiMix$( \mathcal{C})$};
\item we assume additionally {\bf Margin$(\mathcal{K})$} for some $\mathcal
K>0$.
\end{itemize}

Let us provide some examples of processes satisfying the uniform mixing
assumption 
{\bf PhiMix$(
\mathcal{C})$}.
In the three following examples $(\epsilon_t)$ denotes an iid sequence
(called the innovations).
\begin{exm}[AR($p$) process]
Consider the stationary solution $(X_t)$ of an AR($p$) model:
$\forall t\in\Z$,
$
X_t=\sum_{j=1}^pa_{j}X_{t-j}+\epsilon_t
$.
Assume that $(\epsilon_t)$ is
bounded with a distribution possessing an absolutely continuous component.
If $\mathcal A(z)=\sum_{j=1}^pa_jz^j$ has no root inside the unit disk
in $\C$ then   $(X_t)$  is a geometrically
$\phi$-mixing processe, see \cite{Athreya86} and {\bf PhiMix$(\mathcal{C})$}
is satisfied for some $\mathcal C$.
\end{exm}

\begin{exm}[MA($p$) process]
Consider the stationary process $(X_t)$ such that
$X_t=\sum_{j=1}^p b_j\epsilon_{t-j}$ for all $t\in\Z$. By definition, the
process $(X_t)$ is stationary and $\phi$-dependent - it is even $p$-dependent,
in the sense that $\phi_r = 0$ for $r>p$. Thus {\bf PhiMix$( \mathcal{C})$} is  
satisfied for some $\mathcal C>0$.
\end{exm}

\begin{exm}[Non linear processes]
For extensions of the AR($p$) model of the form $
X_t=F(X_{t-1},\ldots,X_{t-p};\epsilon_t)$,
$\Phi$-mixing coefficients can also be computed and satisfy {\bf
PhiMix$(\mathcal{C})$}.
See e.g.
\cite{Meyn1993}.
\end{exm}

We now provide an example of predictive model satisfying all the assumptions
required to obtain fast
rates oracle inequalities, in particular  {\bf Margin$(\mathcal{K})$}, when the
loss function
$\ell$ is quadratic, i.e. $\ell(x,x')=(x-x')^2$:

\begin{exm}
\label{exm_margin}
Consider Example~\ref{ex-nlin} where
$$ f_{\theta}(X_{t-1},\ldots, X_{t-k}) = \sum_{i=1}^{N}
\theta_i \varphi_{i}(X_{t-1},\ldots,X_{t-k}) , $$
for functions $(\varphi_{i})_{i=0}^{\infty}$ of 
$ (\mathbb{R}^{p})^{k} $ to $\mathbb{R}^{p} $,
 and  
$\theta=(\theta_1,\ldots,\theta_\N)\in\mathbb{R}^{N}$. Assume the $\varphi_i$ 
upper bounded by $1$ and  
$\Theta=\{\theta\in\mathbb{R}^{N},\|\theta\|_1 \leq L\}$ such that
 {\bf Lip$(L)$}.  Moreover {\bf LipLoss$(K)$} is
satisfied with $K=2\mathcal{B}$.
Assume that
$\overline{\theta}=\arg\min_{\theta\in\mathbb{R}^{N}} R(\theta) \in\Theta$
in order to have:
\begin{align*}
 \E & 
\left\{   \left[\Bigl(X_{q+1}-f_{\theta}(X_{q},...,X_{1})\Bigr)^2 
         -
      \Bigl(X_{q+1}-f_{\overline{\theta}}(X_{q},...,X_{1})\Bigr)^2 \right]^2
\right\}
\\
 & = \E
\Bigl\{   \left[f_{\theta}(X_{q},...,X_{1}) -
f_{\overline{\theta}}(X_{q},...,X_{1})
              \right]^2 
     \\ & \quad \quad \quad \quad \quad \left[2X_{q+1} -
f_{\theta}(X_{q},...,X_{1})
            - f_{\overline{\theta}}(X_{q},...,X_{1})\right]^2 \Bigr\}
 \\
 & \leq 
\E 
\left\{   \left[f_{\theta}(X_{q},...,X_{1}) -
f_{\overline{\theta}}(X_{q},...,X_{1})
              \right]^2  4 \mathcal{B}^{2} (1+R)^2 \right\}
\\
& \leq 4 \mathcal{B}^{2} (1+R)^2 \left[R(\theta)-R(\overline{\theta})\right]
\text{ by Pythagorean theorem.}
\end{align*}
Assumption {\bf Margin$(\mathcal{K})$} is satisfied with $\mathcal{K}=4
\mathcal{B}^{2} (1+D)^2$. According to Theorem~\ref{thmfastrates} below,
the oracle inequality with fast rates holds as soon as
Assumption {\bf PhiMix$(\mathcal{C})$}
is satisfied.
\end{exm}

\subsection{General result}

We only give oracle inequalities for the Gibbs predictor in the model-selection
setting. In the case of one single model, this result can
be extended to the ERM predictor. For several models, the approach based on the
ERM predictors requires a penalized risk minimization procedure as in the slow
rates case. In the fast rates case, the Gibbs predictor itself directly have
nice properties. Let  $ \Theta = \bigsqcup_{i=1}^{M} \Theta_i $
(disjoint union), choose
$\pi = \sum_{j=1}^{M} p_j \pi_j $ and denote $ \overline{\theta}_j
\in\arg\min_{\theta\in\Theta_j}R(\theta)$ as previously.

\begin{thm}
\label{thmfastrates}
Assume that:
\begin{enumerate}
\item  {\bf Margin$(\mathcal{K})$} and  {\bf LipLoss$(K)$} are satisfied for
some $K$, $\mathcal K>0$;
\item {\bf Bound}$(\mathcal{B})$ is satisfied for some
$\mathcal{B}>0$;
\item {\bf  PhiMix$(\mathcal{B})$} is satisfied for some $\mathcal
C>0$;
\item {\bf Lip}($L$) is satisfied for some $L>0$;
\item for any $j\in\{1,...,M\}$,
       there exist
 $ d_j=d(\Theta_j,\pi)  $ and $D_j=D(\Theta_j,\pi_j)$
 satisfying the relation
$$ \forall \delta>0,\quad
\log \frac{1} {\int_{\theta\in\Theta_j}  \mathbf{1} \{ R(\theta) -
R(\overline{\theta}_j)
    < \delta \}  \pi_j({\rm d}\theta)} 
 \leq d_j\log\left(\frac{D_j}{\delta}\right). $$

\end{enumerate}
Then for
$$ \lambda = \frac{n-k}{4k KL\mathcal{B} \mathcal{C}} \wedge \frac{n-k}{16 k
\mathcal{C}} $$
the oracle inequality \eqref{PAC-Gibbs} for any $\varepsilon>0$  with
\begin{multline*}
\Delta(n,\lambda,\pi,\varepsilon)
 \\
=
4 \inf_{ j} \left\{ R(\overline{\theta}_j)
      - R(\overline{\theta}) 
     + 4 k \mathcal{C} \left(4 \vee KL\mathcal{B}\right)
\frac{d_j \log\left(\frac{D_j e (n-k)}{16 k \mathcal{C} d_j}\right) +
\log\left(\frac{2}{\varepsilon p_j}\right)}{n-k}
\right\}.
\end{multline*}
\end{thm}

Compare with the slow rates case, we
don't have to optimize with respect to $\lambda$ as the optimal
order for $\lambda$ is independent of $j$. In practice, the value of $\lambda$
provided by Theorem
\ref{thmfastrates} is too conservative. In the iid case, it is shown in
\cite{DT1} that the value $\lambda=n/(4\sigma^2)$, where $\sigma^2$ is the
variance
of the noise of the regression yields good results.
In our simulations results, we will use  $\lambda = n/\hat{{\rm var}}(X)$, where
$\hat{{\rm var}}(X)$ is the empirical
variance of the observed time series.

Notice that for the index $j_0$ such that $R(\overline{\theta}_{j_0})
= R(\overline{\theta})$  we obtain:
\begin{multline*}
R\left(\hat{\theta}_{\lambda}\right)
\leq
\\
R(\overline{\theta})
     + 4 k \mathcal{C} \left(4 \vee KL\mathcal{B}\right)
\frac{d_{j_0} \log\left( {c_{j_0} e (n-k)}/({16 k \mathcal{C} d_{j_0}})\right) +
\log\left( {2}/({\varepsilon p_{j_0}})\right)}{n-k}.
\end{multline*}
So, the oracle inequality achieves the fast rate ${d_{j_0}}/{n} \log\left(
{n}/{d_{j_0}}\right) $ where $j_0$ is the
 model of the oracle. However, note that the choice $j=j_0$ does not necessarily
reach the infimum
in Theorem~\ref{thmfastrates}.

Let us compare the rates in Theorem~\ref{thmfastrates} to the
ones in \cite{meir,modha,AGARWAL,AGARWAL2}. In \cite{meir,modha}, the optimal
rate
$1/n$ is never obtained. The paper \cite{AGARWAL} proves fast rates for online
algorithms that
are also computationally efficient, see also \cite{AGARWAL2}. The fast rate
$1/n$
is reached when the coefficients $(\phi_r)$ are geometrically decreasing. In
other cases, the rate is slower.
Note that we do not suffer such a restriction. The Gibbs estimator of
Theorem~\ref{thmfastrates} can also be computed efficiently thanks
to MCMC procedures, see   \cite{AL,DT1}.

\subsection{Corollary: sparse autoregression}

Let  the  predictors be the linear autoregressive predictors 
$$ \hat{X}^{\theta}_{p} = \sum_{j=1}^{p} X_{p-j} \theta_j .$$
For any $J\subset\{1,\dots,p\}$, define the model:
$$ \Theta_{J} = \{\theta\in\mathbb{R}^{p}: \|\theta\|_{1} \leq L \text{ and }
\theta_j \neq 0 \Leftrightarrow j\in J
\} . $$
Let us remark that we have the disjoint union $\Theta =
\bigsqcup_{J\subset\{1,\dots,p\}} \Theta_J
 = \{\theta\in\mathbb{R}^{p}: \|\theta\|_{1} \leq 1 \}.$
 We choose $\pi_J$ as the uniform probability measure on
  $ \Theta_J$ and $p_j=2^{-|J|-1} {p\choose |J|}^{-1} $.
\begin{cor}
\label{corsparse}
Assume that
$\overline{\theta}=\arg\min_{\theta\in\mathbb{R}^{N}} R(\theta) \in\Theta$
and   {\bf  PhiMix$( \mathcal{C})$} is satisfied for some $\mathcal C>0$ as well
as
{\bf Bound}$(\mathcal{B})$. Then the oracle inequality \eqref{PAC-Gibbs} is satisfied  for
any $ \varepsilon>0$ with
$$
\Delta(n,\lambda,\pi,\varepsilon)
=
4 \inf_{ J } \left\{ R(\overline{\theta}_J)
      - R(\overline{\theta}) 
     + {\rm cst}.
\frac{|J| \log\left({(n-k) p}/{|J|} \right) +
\log\left(\frac{2}{\varepsilon}\right)}{n-k}
\right\}
$$
for some constant ${\rm cst} = {\rm cst}(\mathcal{B},\mathcal{C},L)$.
\end{cor}
This extends the results of \cite{AL,DT1,gerchi} to the case
of autoregression.

\begin{proof}
The proof follows the computations of Example~\ref{exm_margin} that we do not
reproduce here:
we  check the conditions 
 {\bf LipLoss$(K)$}  with $K=2\mathcal{B}$,
 {\bf Lip$(L)$} and
 {\bf Margin$(\mathcal{K})$}  with
$\mathcal{K}=4 \mathcal{B}^{2} (1+L)^2$. We can apply Theorem~\ref{thmfastrates}
with $d_J = |J|$ and $D_j=L$.
\end{proof}

\section{Application to French GDP forecasting}
\label{section_application}

\subsection{Uncertainty in GDP forecasting}

Every quarter $t\ge 1$,  the French national bureau of
statistics, INSEE\footnote{{\it Institut
National de la Statistique et des Etudes Economiques}http://www.insee.fr/},
publishes 
the  growth rate of the French GDP (Gross Domestic Product). 
Since it involves a huge amount of data that take months to be collected and
processed,
the computation of the GDP growth rate $\log({\rm
GDP}_{t}/{\rm GDP}_{t-1})$
takes a long time (two years). 
This means
that at time $t$, the value $\log({\rm GDP}_{t}/{\rm GDP}_{t-1})$ is actually
not
known. However, a preliminary value of the growth rate is published
45 days only after the end of the current quarter $t$. This value
is called a {\it flash estimate} and is the quantity that INSEE forecasters
actually try to predict, at least in a first time. As we want to work under
the same constraint as the
INSEE,
we will now focus on the prediction on the flash estimate and let
$\Delta {\rm GDP}_{t}$ denote this quantity.
To forecast at time $t$, we will use:
\begin{enumerate}
 \item the  past forecastings\footnote{It has been checked that to replace
 past flash estimates by the actual GDP growth rate when it becomes available do
 not improve the quality of the forecasting \cite{Minodier}.} $\Delta {\rm
GDP}_{j}$, $0<j<t$;
 \item past  {\it climate indicators} $I_j$, $0<j<t$, based on {\it business
surveys}.
\end{enumerate}
Business surveys are questionnaires of about ten questions sent monthly to a
representative panel of French companies (see \cite{Devilliers} for more
details). As a consequence, these surveys provide informations 
from the  economic decision makers. Moreover, they are available each end of
months and thus can be used to forecast the french GDP. INSEE publishes a
composite indicator, the {\it French
business climate indicator} that summarizes information of the whole
business survey, see \cite{climate,DuboisMichaux}.
Following \cite{CornecCIRET}, let $I_t$ be  the mean of the last three (monthly
based) climate indicators available  for each quarter $t>0$ at the date of
publication of $\Delta {\rm GDP}_{t}$.
All these values (GDP, climate indicator) are available from the INSEE website.
Note that a similar approach is used in other countries, see e.g.
\cite{predBiau} on forecasting the European Union GDP growth thanks to EUROSTATS
data.

In order to provide a
quantification of the uncertainty of the forecasting, associated interval
confidences are usually provided. The ASA and the NBER started using
density forecasts in 1968, while the Central Bank of England and INSEE provide
their
prediction with a {\it fan chart}, see ee \cite{Diebold,Tay} for surveys on
density forecasting
and \cite{Britton} for fan charts.
However, the statistical methodology used is often  crude and, until 2012, the
fan charts provided by the INSEE was based on the  homoscedasticity of the
Gaussian forecasting errors, see   \cite{CornecCIRET,Dowd}.
However,  empirical evidences are  
\begin{enumerate}
 \item the GDP forecasting is more uncertain in a period of crisis or recession;
 \item the forecasting errors are  not symmetrically distributed.
\end{enumerate}

\subsection{Application of Theorem~\ref{thmERM} for the GDP forecasting}
Define $X_t$ as the data observed at time $t$:
$X_t=(\Delta {\rm GDP}_{t},I_t)'\in\mathbb{R}^{2}$.
We use the quantile
loss function (see Example \ref{exm_quantile} page \pageref{exm_quantile}) for
some $0<\tau<1$ of the quantity of interested $\Delta {\rm GDP}_{t}$:
\begin{multline*}
\ell_{\tau}((\Delta {\rm GDP}_{t},I_t),(\Delta' {\rm GDP}_{t},I'_t))
\\
=
\begin{cases}
\tau\left( \Delta {\rm GDP}_{t} - \Delta' {\rm GDP}_{t} \right), & \text{if }
\Delta {\rm GDP}_{t} - \Delta' {\rm GDP}_{t} >0\\
-\left(1-\tau\right)\left( \Delta {\rm GDP}_{t} - \Delta' {\rm GDP}_{t} 
\right), & \text{otherwise}.
\end{cases}
\end{multline*}
We use the family of forecasters proposed by \cite{CornecCIRET} given by the
relation
\begin{equation}
\label{modele}
f_{\theta}(X_{t-1},X_{t-2})=\theta_0+\theta_1 \Delta {\rm GDP}_{t-1} +
\theta_{2} I_{t-1} + \theta_{3}
                      (I_{t-1}-I_{t-2})|I_{t-1} - I_{t-2}|
\end{equation}
where $\theta=(\theta_0,\theta_1,\theta_2,\theta_3)\in \Theta(B)$.
Fix $D>0$ and
$$\Theta=
\biggl\{\theta=(\theta_0,\theta_1,\theta_2,\theta_3)\in\mathbb{R}^{4},
 \|\theta\|_1=\sum_{i=0}^3 |\theta_i| \leq D \biggr\} .$$
Let us denote
$ R^{\tau}(\theta) := \mathbb{E} \left[ \ell_{\tau} \left( \Delta {\rm
GDP}_{t}, f_{\theta}(X_{t-1},X_{t-2})\right) \right]$ the risk of the forecaster
$ f_{\theta}$
and let $r_{n}^{\tau}$ denote the associated empirical risk.
We let $\hat{\theta}^{ERM,\tau}$ denote the ERM with quantile loss
$\ell_{\tau}$:
$$ \hat{\theta}^{ERM,\tau} \in\arg\min_{\theta\in\Theta} r_n ^{\tau}(\theta). $$

We apply Theorem~\ref{thmERM} as  {\bf Lip}$(L)$ is satisfied $\Theta'$ with $L
= D+1$ and {\bf LipLoss$(K)$}  with $K=1$. If the observations are bounded,
stationary such that {\bf WeakDep$(\mathcal{C})$} holds for some
$\mathcal{C}>0$,
the assumptions
of Theorem~\ref{thmERM} are satisfied with $\psi=\mathcal{B}$ and $d=4$:

\begin{cor}
\label{thm_appli}
Let us fix $\tau\in(0,1)$.
If the observations are bounded, stationary such that  {\bf
WeakDep$(\mathcal{C})$} holds for some $\mathcal{C}>0$ then for any
$\varepsilon>0$ and $n$ large enough, we have
\begin{multline*}
\mathbb{P} \left\{
 R^{\tau}(\hat{\theta}^{ERM,\tau})
 \leq \inf_{\theta\in\Theta} R^{\tau}(\theta)
      + \frac{2\kappa\sqrt{2}}{\sqrt{n}\left(1- {4}/{n}\right) }
    \log\left(\frac{2e^2 \mathcal{B} (D+1) \sqrt{n} }{\kappa\varepsilon}\right)
\right\} \\
\geq 1-\varepsilon.
\end{multline*}
\end {cor}

In practice  the choice of $D$ has little importance as soon as $D$ is
large enough (only the theoretical bound is influenced).
As a consequence we take $D=100$ in our experiments.

\subsection{Results}

The results are shown in
Figure \ref{fig05} for forecasting corresponding to $\tau=0.5$. Figure
\ref{fig025} represents the
confidence intervals of order $50\%$, i.e. $\tau=0.25$ and $\tau=0.75$ (left)
and for
confidence interval of order $90\%$, i.e. $\tau=0.05$ and $\tau=0.95$ (right).
We report only the results for the period 2000-Q1 to 2011-Q3 (using the period
1988-Q1 to 1999-Q4 for learning).

\begin{figure}[!!h]
\begin{center}
\centering
\includegraphics*[height=5cm,width=7cm]{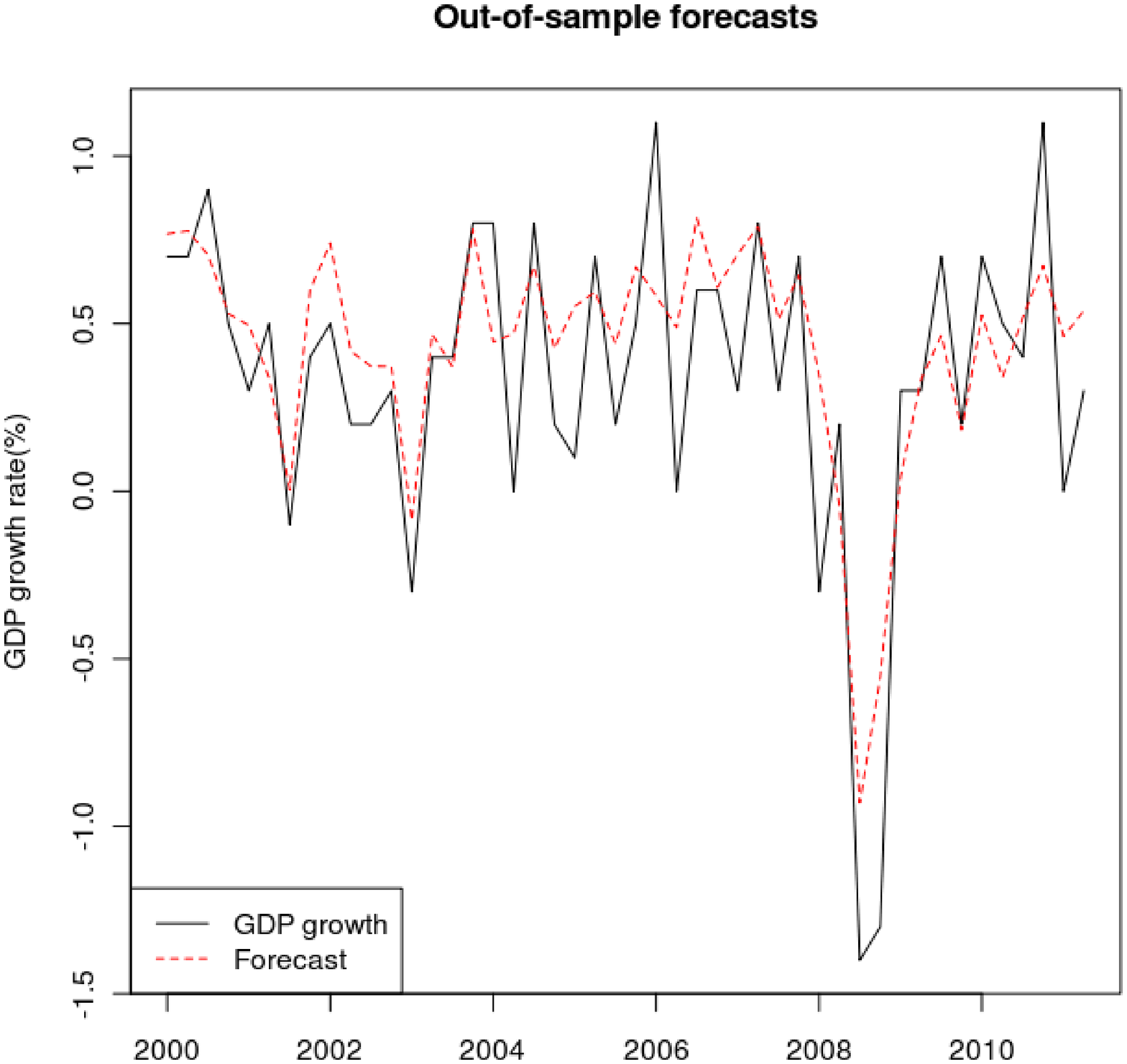}
\caption{\label{fig05} \textsf{French GDP forecasting using the quantile
loss function with $\tau=0.5$.}}
\end{center}
\end{figure}

\begin{figure}[!!h]
\begin{center}
\centering
\begin{tabular}{c c} 
\includegraphics*[width=5cm]{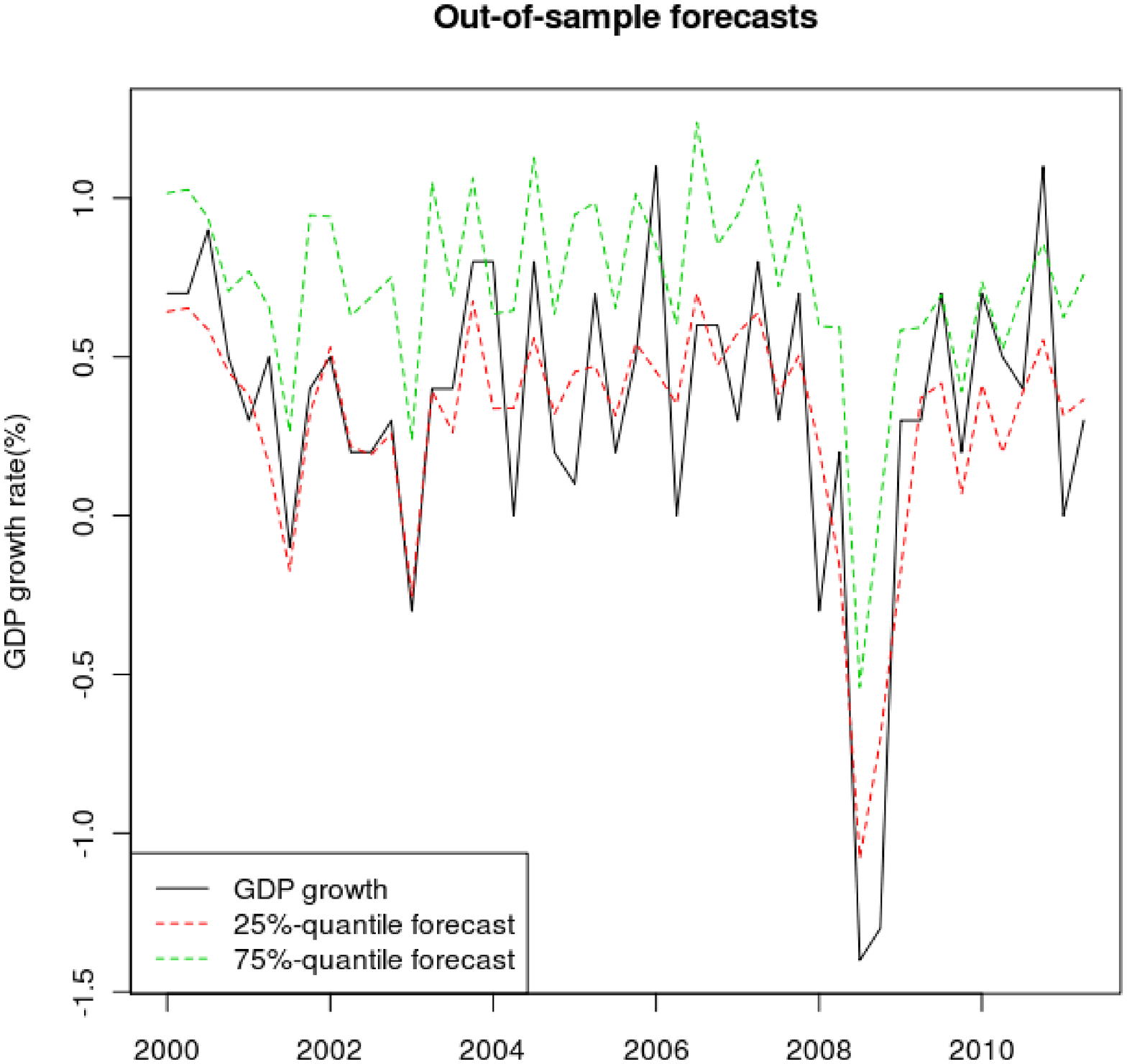} & \includegraphics*[width=5cm]{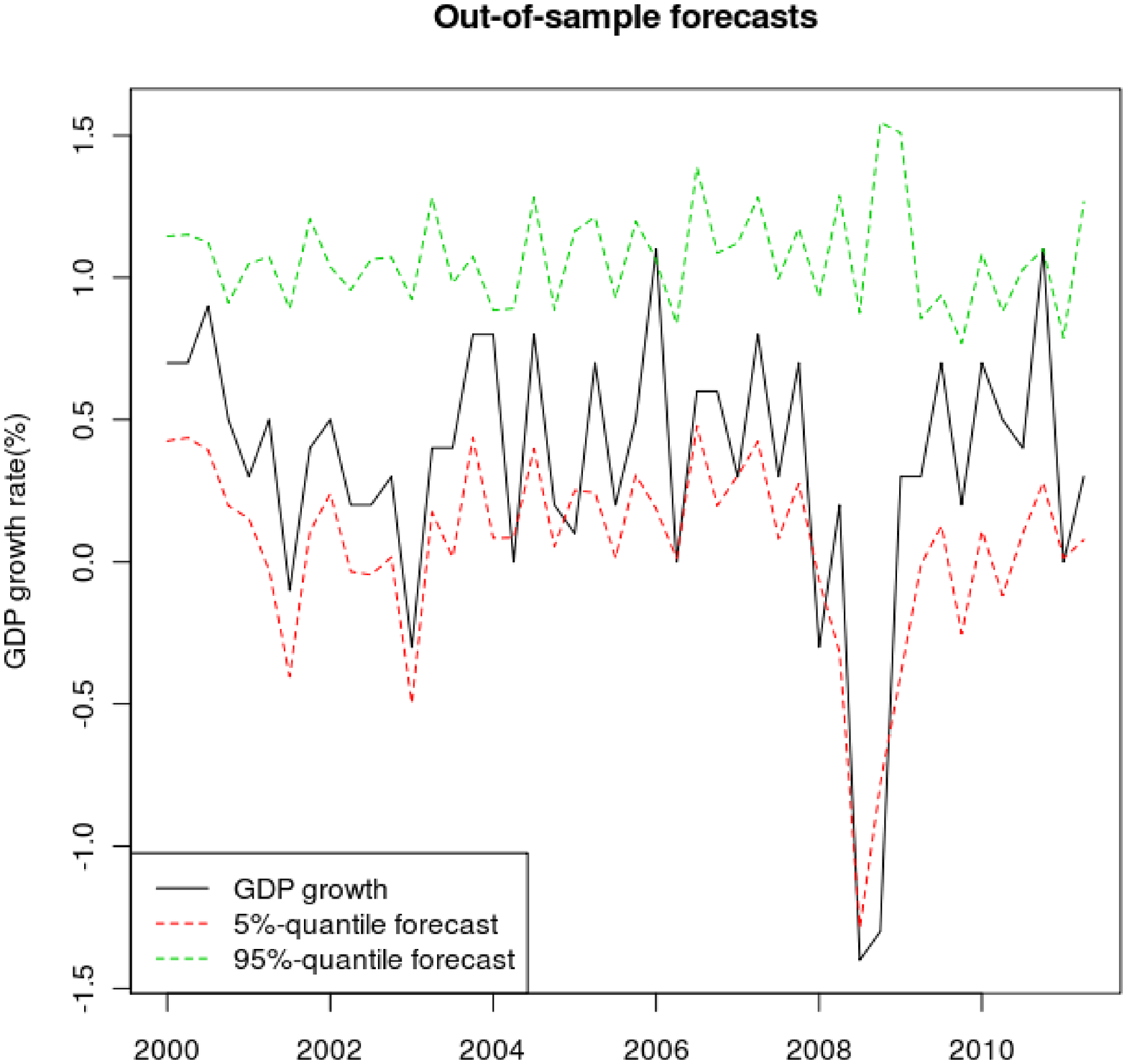}
\end{tabular}
\caption{\label{fig025} \textsf{French GDP online $50\%$-confidence intervals
(left)
and $90\%$-confidence intervals (right).}}
\end{center}
\end{figure}

We denote $\hat{\theta}^{ERM,\tau}[t]$ the estimator computed at time $t-1$,
based
on the observations $X_j$, $j<t$. We report the online performance:
$$
\begin{array}{c c}
\text{mean abs. pred. error} &=  \frac{1}{n}\sum_{t=1}^{n} \left| \Delta GDP_{t}
-
f_{\hat{\theta}^{ERM,0.5}[t]}(X_{t-1},X_{t-2}) \right|
\\
\text{mean quad. pred. error} &= \frac{1}{n}\sum_{t=1}^{n} \left[ \Delta GDP_{t}
-
f_{\hat{\theta}^{ERM,0.5}[t]}(X_{t-1},X_{t-2}) \right]^2
\end{array}
$$
and compare it to the INSEE performance, see Table~\ref{tableinsee}. We also
report the frequency that the GDPs fall above the
predicted $\tau$-quantiles for each $\tau$, see Table~\ref{tablefreq}.
Note that this quantity should be close to $\tau$.

\begin{table}
\begin{center}
\begin{tabular}{|c|c|c|}
\hline
 Predictor & Mean absolute prediction error & Mean quadratic prediction error \\
 \hline
$\widehat{\theta}^{ERM,0.5}$ & $0.2249$ & $0.0812$ \\ 
INSEE & $0.2579 $ & $ 0.0967 $ \\
\hline
\end{tabular}
\caption{{\it Performances of the ERM and of the INSEE.}}
\label{tableinsee}
\begin{tabular}{|c|c|c|}
\hline 
$\tau$ & Estimator & Frequency \\\hline
$0.05$ & $\widehat{\theta}^{ERM,0.05}$ & $0.1739$ \\
$0.25$ & $\widehat{\theta}^{ERM,0.25}$ & $0.4130$ \\
$0.5$  & $\widehat{\theta}^{ERM,0.5}$  & $0.6304$ \\
$0.75$ & $\widehat{\theta}^{ERM,0.75}$ & $0.9130$ \\
$0.95$ & $\widehat{\theta}^{ERM,0.95}$ & $0.9782$ \\
\hline
\end{tabular}
\caption{{\it Empirical frequencies of the event: GDP falls under the predicted
  $\tau$-quantile.}}
\label{tablefreq}
\end{center}
\end{table}

The methodology fails to forecast the importance of the 2008 subprime crisis as
it was the case for the INSEE forecaster, see
\cite{CornecCIRET}.
However, it is interesting to note that the confidence
interval is larger at that date: the forecast is less reliable,
but thanks to our adaptive confidence interval, it would
have been possible to know at that time that the prediction was
not reliable.
Another interesting point is
to remark that the lower bound of the  
confidence intervals are  varying over time  while the upper bound is almost
constant for $\tau=0.95$. It supports the idea of asymmetric forecasting errors.
A parametric model with gaussian innovations would
lead to underestimate the  recessions risk.

\section{Simulation study}
\label{section_simulation}

In this section, we finally compare the ERM or Gibbs estimators to the
Quasi Maximum Likelihood Estimator (QMLE) based method used by the R function
ARMA \cite{R}.
The idea is not to claim any superiority of one method over another, it is
rather to check that
the ERM and Gibbs estimators can be safely used in various contexts as their
performances
are close to the standard QMLE even in the context where the series is generated
from an ARMA
model. It is also the opportunity to check the robustness of our estimators in
case
of misspecification.

\subsection{Parametric family of predictors}

Here, we compare the ERM to the QMLE.

We draw simulations from an AR(1) models  \eqref{ar1} and a non linear model
\eqref{sinar1}:
\begin{eqnarray}
\label{ar1}
X_t = 0.5 X_{t-1} + \varepsilon_t
\end{eqnarray}
\begin{eqnarray}
\label{sinar1}
X_{t} = 0.5 \sin (X_{t-1}) + \varepsilon_t
\end{eqnarray}
where $\varepsilon_t$ are iid innovations. We consider two cases
of distributions for $\varepsilon_t$:
the uniform 
case,  $\varepsilon_t \sim \mathcal{U}[-a, a] $, and the Gaussian case, 
$\varepsilon_t \sim \mathcal{N}(0,\sigma^2)$. Note that, in the first case, both
models
satisfy
the assumptions of Theorem~\ref{thmfastrates}:
there exists
a stationary solutions $(X_t)$ that is $\phi$-mixing when the
innovations are uniformly distributed and {\bf  WeakDep$( \mathcal{C})$} is
satisfied for
some $\mathcal{C}>0$. This paper does not provide any theoretical results for
the Gaussian
case as it is unbounded. However, we refer the reader to \cite{alqwin} for
truncations
techniques that allows to deal with this case too. We fix $\sigma=0.4$ and
$a=0.70$ such
that $Var(\epsilon_t) \simeq 0.16$ in both cases. 
For each model, we simulate first a sequence of length $n$ and then we predict
$X_n$ using
the observations $(X_1,\ldots,X_{n-1})$.
Each simulation is repeated $100$ times and we report the mean quadratic
prediction errors on
the Table \ref{tableXiaoyin}.
 
\begin{table} 
\begin{center}
\begin{tabular}{|c|c|c|c|c|c|}
\hline 
n & Model & Innovations & ERM abs.  & ERM quad. & QMLE  \\ \hline
\hline
 $100$ &  \eqref{ar1} & Gaussian &  {\bf 0.1436}  (0.1419) & 0.1445 (0.1365)  &  0.1469 (0.1387)\\
  &   &  Uniform & 0.1594  (0.1512) &{\bf 0.1591}(0.1436) &0.1628 (0.1486) \\  \hline
  & \eqref{sinar1}  &  Gaussian &0.1770 (0.1733) & {\b 0.1699} (0.1611) &0.1728 (0.1634) \\
  &    &  Uniform & {\bf 0.1520} (0.1572) &0.1528 (0.1495)&  0.1565 (0.1537) \\  \hline
 \hline
 $1000$ &  \eqref{ar1} & Gaussian &  {\bf 0.1336} (0.1291) & 0.1343 (0.1294) & 0.1345 (0.1296) \\
 &    &  Uniform & {\bf 0.1718} (0.1369) &0.1729 (0.1370) & 0.1732 (0.1372) \\  \hline
 & \eqref{sinar1}   &  Gaussian & 0.1612( 0.1375) &{\bf 0.1610} (0.1367) &0.1613 (0.1369) \\
 &    &  Uniform &0.1696 (0.1418) & {\bf 0.1687} (0.1404)& 0.1691 (0.1407) \\  
 \hline 
\end{tabular}
\end{center}
\caption{\textit{Performances of the ERM estimators and ARMA, on the
simulations.
The first row ``ERM abs.''
is for the ERM estimator with absolute loss, the second row ``ERM quad.''
for the ERM with quadratic loss. The standard deviations are given in parentheses.}}
\label{tableXiaoyin}
\end{table}

It is interesting to note that the ERM estimator with absolute loss performs
better
on model \eqref{ar1} while the ERM with quadratic loss performs slightly better
on
model \eqref{sinar1}. The difference tends be too small to be significative,
however, the numerical results tends to indicate that both methods are robust
to model mispecification. Also, both estimators seem to perform better than
the R QMLE procedure when $n=100$, but the differences tends to 
be less perceptible when $n$ grows.

\subsection{Sparse autoregression}

To illustrate Corollary~\ref{corsparse},
we compare the Gibbs predictor to the model
selection approach
of the ARMA procedure in the R software.
This procedure computes the QMLE estimator in each
AR$(p)$ model, $1\le p\le q$, and then selects the order $p$ by Akaike's AIC
criterion \cite{aic}.
The Gibbs estimator  is computed using a Reversible Jump MCMC algorithm as
in \cite{AL}. The parameter $\lambda$ is taken as $\lambda = n/\hat{{\rm
var}}(X)$,
the empirical variance of the observed time series.

We draw the data according to the following models:
\begin{align}
\label{align1}
X_t & = 0.5 X_{t-1} + 0.1 X_{t-2} + \varepsilon_{t}
\\
\label{align2}
X_{t} & = 0.6 X_{t-4} + 0.1 X_{t-8} + \varepsilon_{t}
\\
\label{align3}
X_{t} & = \cos (X_{t-1})\sin(X_{t-2}) + \varepsilon_{t}
\end{align}
where $\varepsilon_{t}$ are iid innovations. We still consider the
 uniform ($\varepsilon_{t}\sim\mathcal{U}[-a,a]$) and the Gaussian
($\varepsilon_t \sim \mathcal{N}(0,\sigma^{2})$) cases with $\sigma=0.4$ and
$a=0.70$.
We compare the Gibbs predictor performances to those of
the estimator based on the AIC criterion and to the QMLE
in the $AR(q)$ model, so called ``full model''. For each model, we first
simulate a time series of length $2n$, use the observations $1$ to $n$ as
a learning set and $n+1$ to $2n$ as a test set, for $n=100$ and $n=1000$. Each
simulation is repeated 20 times and we report
in Table~\ref{tablesimu} the mean  and the standard deviation of the empirical
quadratic errors for each method and each model.

\begin{table}[t!]
\caption{{\it Performances of the Gibbs, AIC and ``full model'' predictors   on
simulations.}
}
\label{tablesimu}
\begin{center}
\begin{tiny}
\begin{tabular}{|p{1.0cm}|p{1.0cm}|p{1.5cm}||p{1.6cm}|p{1.6cm}|p{1.6cm}|}
\hline
 $n$ & Model & Innovations & Gibbs & AIC & Full Model \\
\hline \hline
 $100$ & \eqref{align1} & Uniform  & { 0.165} (0.022) & { 0.165} (0.023) & 0.182
(0.029) \\
       &                & Gaussian & 0.167 (0.023) & { 0.161} (0.023) & 0.173
(0.027) \\
\hline
       & \eqref{align2} & Uniform  & { 0.163} (0.020) & 0.169 (0.022) & 0.178
(0.022) \\
       &                & Gaussian & { 0.172} (0.033) & 0.179 (0.040) & 0.201
(0.049) \\
\hline
       & \eqref{align3} & Uniform  & { 0.174} (0.022) & 0.179 (0.028) & 0.201
(0.040) \\
       &                & Gaussian & { 0.179} (0.025) & 0.182 (0.025) & 0.202
(0.031) \\
\hline \hline
 $1000$& \eqref{align1} & Uniform  & { 0.163} (0.005) & { 0.163} (0.005) & 0.166
(0.005) \\
       &                & Gaussian & { 0.160} (0.005) & { 0.160} (0.005) & 0.162
(0.005) \\
\hline
       & \eqref{align2} & Uniform  & { 0.164} (0.004) & 0.166 (0.004) & 0.167
(0.004) \\
       &                & Gaussian & { 0.160} (0.008) & 0.161 (0.008) & 0.163
(0.008) \\
\hline
       & \eqref{align3} & Uniform  & { 0.171} (0.005) & 0.172 (0.006) & 0.175
(0.006) \\
       &                & Gaussian & { 0.173} (0.009) & { 0.173} (0.009) & 0.176
(0.010) \\
\hline
\end{tabular}
\end{tiny}
\end{center}
\vspace*{-6pt}
\end{table}

Note that the Gibbs predictor performs better on Models~\eqref{align2}
and~\eqref{align3} while the AIC predictor performs slightly better on
Model~\eqref{align1}.
The difference tends to
be negligible when $n$ grows - this is coherent with the fact that we develop
here a non-asymptotic theory. Note that the Gibbs predictor performs also well 
in the case of a Gaussian noise where the boundedness assumption is not
satisfied.

\appendix

\section{A general PAC-Bayesian inequality}
\label{sectionGENERALPACBAYES}

Theorems~\ref{corofinite} and~\ref{thmGibbs1} are actually both corollaries
of a more general result that we would like to state for the sake of
completeness.
This result is the analogous of the PAC-Bayesian bounds proved by Catoni
in the case of iid data~\cite{Catoni2007}.

\begin{thm}[PAC-Bayesian Oracle Inequality for the Gibbs estimator]
\label{main_result}
Let us assume that {\bf LowRates($\kappa$)} is satisfied for some $\kappa>0$.
Then, for any $\lambda$, $\varepsilon>0$ we have
\begin{multline*}
\mathbb{P} \left\{ R\left(\hat{\theta}_{\lambda}\right)  \leq
\inf_{\rho\in\mathcal{M}_{+}^{1}(\Theta)} \left[ \int R {\rm d}\rho
 + \frac{2\lambda \kappa^{2}}{n\left(1- {k}/{n}\right)^2} +
\frac{2\mathcal{K}(\rho,\pi) +  2 \log\left( {2}/{\varepsilon}\right)
}{\lambda} \right]
\right\}
\\
\geq 1-\varepsilon.
\end{multline*}
\end{thm}
This result is proved in Appendix~\ref{sectionproofs}, but we can now provide
the proofs of
Theorems~\ref{corofinite} and~\ref{thmGibbs1}.

\noindent {\it Proof of Theorem~\ref{corofinite}.}
We apply
Theorem~\ref{main_result} for $\pi= \frac1M\sum_{\theta \in
\Theta}\delta_{\theta}$ and restrict the $\inf$ in the upper bound to
Dirac masses $\rho\in\{\delta_{\theta},\theta\in\Theta\}$. We obtain
$ \mathcal{K}(\rho,\pi)=\log M$, and the upper bound for 
$ R (\hat{\theta}_{\lambda} ) $  becomes:
\begin{align*}
  R\left(\hat{\theta}_{\lambda}\right) &  \leq
\inf_{\rho \in \{ \delta_{\theta}, \theta \in \Theta \}} \left[ \int R {\rm
d}\rho
 + \frac{2\lambda \kappa^{2}}{n\left(1-{k}/{n}\right)^2} +
\frac{ 2 \log\left({2M}/{\varepsilon}\right)
}{\lambda} \right]
 \\
 & = \inf_{\theta\in\Theta} \left[
R(\theta) + \frac{2\lambda \kappa^{2}}{n\left(1-{k}/{n}\right)^2} +
\frac{ 2 \log\left({2M}/{\varepsilon}\right)
}{\lambda}
 \right].
\end{align*}
\hfill $\blacksquare$

\noindent {\it Proof of Theorem~\ref{thmGibbs1}.}
An application of Theorem~\ref{main_result} yields that with probability at
least $1-\varepsilon$ 
$$
R (\hat{\theta}_{\lambda} )  \leq
\inf_{\rho\in\mathcal{M}_{+}^{1}(\Theta)} \left[ \int R {\rm d}\rho
 + \frac{2\lambda \kappa^{2}}{n\left(1- {k}/{n}\right)^2} +
\frac{2\mathcal{K}(\rho,\pi) +  2 \log\left( {2}/{\varepsilon}\right)
}{\lambda} \right].
$$
Let us estimate the upper bound at the probability distribution $ \rho_{\delta}
$ defined as
$$ \frac{ {\rm d} \rho_{\delta} }{{\rm d}\pi}(\theta) = \frac{\mathbf{1} \{
R(\theta) - R(\overline{\theta})< \delta \}}
                         { \int_{t\in\Theta}  \mathbf{1} \{ R(t) -
R(\overline{\theta})< \delta \}  \pi({\rm d}t)     }.  $$
Then we have:
\begin{multline*}
R\left(\hat{\theta}_{\lambda}\right)  \leq
\inf_{\delta>0 } \Biggl[R(\overline{\theta}) + \delta
 + \frac{2\lambda \kappa^{2}}{n\left(1- {k}/{n}\right)^2} 
\\
+
2 \frac{ - \log \int_{t\in\Theta}  \mathbf{1} \{ R(t) - \inf_{\Theta} R < \delta
\}  \pi({\rm d}t)
  +   \log\left(\frac{2}{\varepsilon}\right)
}{\lambda} \Biggr].
\end{multline*}
Under the assumptions of Theorem~\ref{thmGibbs1} we have:
$$
R\left(\hat{\theta}_{\lambda}\right)  \leq
\inf_{\delta>0 } \Biggl[ R(\overline{\theta}) + \delta
 + \frac{2\lambda \kappa^{2}}{n\left(1- {k}/{n}\right)^2} 
+
2 \frac{ d\log \left( {D}/{\delta}\right)
  +   \log\left(\frac{2}{\varepsilon}\right)
}{\lambda} \Biggr].
$$
The infimum is reached for $\delta = d/\lambda$ and we have:
$$
R\left(\hat{\theta}_{\lambda}\right)  \leq
R(\overline{\theta})
 + \frac{2\lambda \kappa^{2}}{n\left(1- {k}/{n}\right)^2} 
+
2 \frac{ d\log \left({D\sqrt{e} \lambda}/{d}\right)
  +   \log\left(\frac{2}{\varepsilon}\right)
}{\lambda} .
$$
\hfill $\blacksquare$

\section{Proofs}
\label{sectionproofs}

\subsection{Preliminaries}

We will use Rio's inequality \cite{Rio2000a} that is an extension of Hoeffding's
inequality in a dependent context.
 For the sake of completeness, we
provide here this result when the observations $(X_1,\ldots,X_n)$ come from a
stationary process $(X_t)$
\begin{lemma}[Rio \cite{Rio2000a}]
\label{RIO}
Let $h$ be a function $(\mathbb{R}^{p})^{n}\rightarrow\mathbb{R}$ such that for
all $x_1$,
..., $x_n$, $y_1$, ..., $y_n\in\mathbb{R}^{p}$,
\begin{equation}
\label{lipcondrio}
         |h(x_1,\ldots,x_n)- h(y_1,\ldots,y_n)|
\leq \sum_{i=1}^{n} \|x_i-y_i\|.
\end{equation}
Then, for any $t>0$, we have
$$
\mathbb{E}\left(\exp({t\left\{\mathbb{E}\left[h(X_{1},\ldots,X_{n})\right]
 - h(X_{1},\ldots,X_{n}) \right\}})\right)
\leq \exp\Big({\frac{t^{2} n \left(\mathcal{B} +
\theta_{\infty,n}(1)\right)^{2}}{2}}\Big)
.
$$
\end{lemma}
Others exponential inequalities can be used
to obtain PAC-Bounds in the context of time series: the inequalities in
\cite{Doukhan1994,Samson}
for mixing time series, and \cite{Dedecker2007a,devdep} under weakest ``weak
dependence''
assumptions, \cite{Seldin} for martingales.  Lemma~\ref{RIO} is
very general and yields optimal low rates of convergence.
For  fast rates of convergence, we will  use Samson's inequality that is an
extension of Bernstein's inequality in a dependent context.
\begin{lemma}[Samson \cite{Samson}]
\label{lapmx}
Let $N\ge 1$, $(Z_i)_{i\in\mathbb{Z}}$ be a stationary process on $\mathbb{R}^k$
and $ \phi_{r}^Z $ denote
its
$\phi$-mixing coefficients. For any measurable
function $f:\mathbb{R}^k\rightarrow [-M,M]$, any 
$0\le t\le 1/(M K_{\phi^Z}^2)$, we have
\begin{equation*}
 \E(\exp(t( S_N(f)-\E S_N(f))))\le \exp\Big(8K_{\phi^Z} N \sigma^2(f)t^2\Big),
\end{equation*}
where  $ S_N (f) := \sum_{i=1}^{N} f(Z_i)$, $K_{\phi^Z}=1+\sum_{r=1}^N
\sqrt{\phi^{Z}_r}$ and
$ \sigma^{2}(f) = {\rm Var}(f(Z_i))$.
\end{lemma}
\noindent {\it Proof of Lemma~\ref{lapmx}.}
This result can be deduced easily from the proof of Theorem 3 of \cite{Samson}
which states a more general result on empirical processes. In page 457 of
\cite{Samson},
replace the definition of $f_N (x_1,\dots,x_n)$ by
$f_N (x_1,\dots, x_n) = \sum_{i=1}^{n} g(x_i) $ (following the notations
of \cite{Samson}). Then check that all the arguments of the proof
remain valid, the claim of Lemma~\ref{lapmx} is obtained
page 460, line 7.
\hfill $\blacksquare$\\

We also remind the variational formula of the Kullback divergence.  
\begin{lemma}[Donsker-Varadhan \cite{Donsker1975} variational formula]
\label{LEGENDRE}
For any $\pi\in\mathcal{M}_{+}^{1}(E)$,
for any measurable upper-bounded function $h:E\rightarrow\mathbb{R}$ we have:
\begin{equation} \label{lemmacatoni2}
\int\exp
(h)d\pi=\exp\left(\sup_{\rho\in\mathcal{M}_{+}^{1}(E)}\biggl(\int h
d\rho-\mathcal{K}(\rho,\pi)\biggr)\right).
\end{equation}
Moreover, the supremum with respect to $\rho$ in the right-hand side is
reached for the Gibbs measure
$\pi\{h\}$ defined by $ \pi\{h\}({\rm d}x) =  e^{h(x)} \pi({\rm d}x) /
\pi[\exp(h)]$.
\end{lemma}
Actually, it seems that in the case of discrete probabilities, this
result was already known by Kullback (Problem 8.28 of Chapter 2 in
\cite{kullback}).
For a complete proof of this variational formula, even in the non integrable
cases, we refer the reader
to \cite{Donsker1975,Catoni2003,Catoni2007}.

\subsection{Technical lemmas for the proofs of Theorems \ref{thfinite}, 
\ref{thmERM}, \ref{thmGibbs2} and \ref{main_result}}

\begin{lemma}
\label{XIAOYIN}
We assume that {\bf LowRates($\kappa$)} is satisfied for some $\kappa>0$. For
any $\lambda>0$ and $\theta\in\Theta$ we have
$$ \mathbb{E}\Big( e^{\lambda (R(\theta) - r_n (\theta))} \Big) \vee {E}\Big( e^{\lambda
(r_n(\theta) - R (\theta)) } \Big) \leq
\exp\Big({\frac{\lambda^{2} \kappa^{2}}{n\left(1-{k}/{n}\right)^2}}\Big) .$$
\end{lemma}
\noindent {\it Proof of Lemma~\ref{XIAOYIN}.}
Let us fix $\lambda>0$ and $\theta\in\Theta$. Let us define the function $h$ by:
$$
h(x_1, \ldots , x_n)=\frac{1}{K(1+L)}\sum_{i=k+1}^n
\ell(f_{\theta}(x_{i-1},\ldots,
x_{i-k}),x_i).
$$
We now check that $h$ satisfies~\eqref{lipcondrio}, remember that $\ell(x,x')=
g(x-x')$ so
\begin{align*}
 & \Bigl|  h\left(  x_{1},\ldots, x_{n}\right)  - h\left(y_{1},\ldots
y_{n}\right)\Bigr| \\
 & \quad \leq \frac{1}{K(1+L)}\sum_{i=k+1}^{n} \Bigl|
g(f_{\theta}(x_{i-1},\ldots,x_{i-k})-x_{i})
                                 -g(f_{\theta}(y_{i-1},\ldots,y_{i-k})-y_{i})
\Bigr| \\ 
 & \quad \leq \frac{1}{1+L}\sum_{i=k+1}^{n} \Bigl\|
\bigl(f_{\theta}(x_{i-1},\ldots,x_{i-k})-x_{i}\bigr)                              
-\bigl(f_{\theta}(y_{i-1},\ldots,y_{i-k})-y_{i}\bigr) \Bigr\|
\end{align*}
where we used Assumption {\bf LipLoss$(K)$} for the last inequality. So we have
\begin{align*}
 & \Bigl|  h\left(  x_{1},\ldots, x_{n}\right)  - h\left(y_{1},\ldots
y_{n}\right)\Bigr| \\
 & \quad \leq \frac{1}{1+L}\sum_{i=k+1}^{n} \biggl( \Bigr\|
f_{\theta}(x_{i-1},\ldots,x_{i-k})
                        - f_{\theta}(y_{i-1},\ldots,y_{i-k}) \Bigr\|  + \Bigl\|
x_i - y_i \Bigr\| \biggr) \\
 & \quad \leq \frac{1}{1+L}\sum_{i=k+1}^{n} \left( \sum_{j=1}^{k} a_{j}(\theta)
                                        \| x_{i-j} - y_{i-j}\|
                 +  \| x_i - y_i \| \right) \\
 & \quad \leq \frac{1}{1+L}\sum_{i=1}^{n} \left(1 +
             \sum_{j=1}^{k}a_{j}(\theta) \right) \| x_i - y_i \| 
 \leq \sum_{i=1}^{n} \|x_i-y_i \|
 \end{align*}
where we used Assumption {\bf Lip}$(L)$. So we can apply Lemma~\ref{RIO} with
$h(X_1,\ldots,X_n)=\frac{n-k}{K(1+L)}
r_{n}(\theta)$, $\mathbb{E}(h(X_1,\ldots,X_n))=\frac{n-k}{K(1+L)} R(\theta)$,
and
$t=K(1+L)\lambda/(n-k)$:
\begin{multline*}
\mathbb{E}\left(e^{\lambda\left[R(\theta)
 - r_{n}(\theta) \right]}\right)
\leq \exp\Big({\frac{\lambda^{2} K^{2} (1+L)^{2}
    \left(\mathcal{B} + \theta_{\infty,n}(1)\right)^{2}}{2 n
    \left(1- {k}/{n}\right)^{2} }}\Big)\\
\leq \exp\Big({\frac{\lambda^{2}K^{2} (1+L)^{2}
    \left(\mathcal{B} + \mathcal{C}\right)^{2}}{2 n
    \left(1-\frac{k}{n}\right)^{2} }}\Big)
\end{multline*}
by Assumption {\bf WeakDep$( \mathcal{C})$}. This ends the proof of the first
inequality.
The reverse inequality is obtained by replacing the function $h$ by $-h$.
\hfill $\blacksquare$\\

We are now ready to state the following key Lemma.
\begin{lemma}
\label{PACBAYES}
Let us assume that {\bf LowRates$(\kappa)$} is
satisfied satisfied for some $\kappa>0$. Then for any $\lambda>0$ we have
\begin{equation}
\label{eqPACBAYES}
\mathbb{P} \left\{
\begin{array}{l}
\forall \rho \in\mathcal{M}_{+}^{1}(\Theta),\\
 \int R {\rm d} \rho \leq \int r_n {\rm d}\rho
       + \frac{\lambda \kappa^2 }{n \left(1- {k}/{n}\right)^{2}}
     + \frac{\mathcal{K}(\rho,\pi) +
\log\left( {2}/{\varepsilon}\right)}{\lambda}
\\
\text{ and }
\\
 \int r_n {\rm d} \rho \leq \int R {\rm d}\rho
       + \frac{\lambda \kappa^2 }{n \left(1- {k}/{n}\right)^{2}}
     + \frac{\mathcal{K}(\rho,\pi) +
\log\left( {2}/{\varepsilon}\right)}{\lambda}
\end{array}
\right\}
\geq 1-\varepsilon.
\end{equation}
\end{lemma}
\noindent {\it Proof of Lemma~\ref{PACBAYES}.}
 Let us fix $\theta>0$ and $\lambda>0$, and apply the first inequality of
Lemma~\ref{XIAOYIN}. We have:
$$
\mathbb{E}\Big(\exp\Big({\lambda \Big(R(\theta) - r_n (\theta) - \frac{\lambda
\kappa^{2}}{n\left(1- {k}/{n}\right)^2}\Big) }\Big)\Big)
          \leq 1,
$$
and we multiply this result by $\varepsilon/2$ and integrate it with
respect to $\pi({\rm d}\theta)$. An application of Fubini's Theorem yields
$$
\mathbb{E}\int \exp\Big({\lambda (R(\theta) - r_n (\theta))
- \frac{\lambda^2 \kappa^{2}}{n\left(1-{k}/{n}\right)^2} - \log
\left({2}/{\varepsilon}\right)}\Big)
       \pi({\rm d}\theta) \leq \frac{\varepsilon}{2}.
$$
We apply Lemma~\ref{LEGENDRE} and we get:
$$
\mathbb{E} \exp\Big({\sup_{\rho} \Big\{ \lambda \int(R(\theta) - r_n
(\theta)) \rho({\rm d}\theta)
 - \frac{\lambda^2 \kappa^{2}}{n\left(1-{k}/{n}\right)^2} - \log
\left({2}/{\varepsilon}\right) - \mathcal{K}(\rho,\pi)
 \Big\}  } \Big) \leq \frac{\varepsilon}{2}.
$$
As $e^{x}\geq \mathbf{1}_{\mathbb{R}_{+}}(x)$, we have:
$$
\mathbb{P} \left\{
\sup_{\rho} \left\{ \lambda \int \left(R(\theta) - r_n (\theta)\right) \rho({\rm
d}\theta)
 - \frac{\lambda^2 \kappa^{2}}{n\left(1- {k}/{n}\right)^2} - \log
\left( {2}/{\varepsilon}\right) - \mathcal{K}(\rho,\pi)
 \right\} \geq 0
 \right\} \leq \frac{\varepsilon}{2}.
$$
Using the same arguments than above but starting with the second inequality of
Lemma~\ref{XIAOYIN}:
$$
\mathbb{E}\exp\Big({\lambda \Big(r_n (\theta) - R (\theta) - \frac{\lambda
\kappa^{2}}{n\left(1- {k}/{n}\right)^2}\Big) }\Big)\Big)
          \leq 1.
$$
we obtain:
$$
\mathbb{P} \left\{
\sup_{\rho} \left\{ \lambda \int \left[r_n(\theta) - R (\theta)\right] \rho({\rm
d}\theta)
 - \frac{\lambda^2 \kappa^{2}}{n\left(1-\frac{k}{n}\right)^2} - \log
\left(\frac{2}{\varepsilon}\right) - \mathcal{K}(\rho,\pi)
 \right\} \geq 0
 \right\} \leq \frac{\varepsilon}{2}.
$$
A union bound ends the proof.
\hfill $\blacksquare$\\

The following variant of Lemma~\ref{PACBAYES} will also be useful.
\begin{lemma}
\label{PACBAYES2}
Let us assume that {\bf LowRates$(\kappa)$} is
satisfied satisfied for some $\kappa>0$. Then for any $\lambda>0$ we have
\begin{equation*}
\mathbb{P} \left\{
\begin{array}{l}
\forall \rho \in\mathcal{M}_{+}^{1}(\Theta), \\
 \int R {\rm d} \rho \leq \int r_n {\rm d}\rho
       + \frac{\lambda \kappa^2 }{n \left(1-{k}/{n}\right)^{2}}
     + \frac{\mathcal{K}(\rho,\pi) +
\log\left({2}/{\varepsilon}\right)}{\lambda}
\\
\text{ and }
\\
 r_n(\overline{\theta}) \leq R(\overline{\theta})
       + \frac{\lambda \kappa^2 }{n \left(1-{k}/{n}\right)^{2}}
     + \frac{
\log\left( {2}/{\varepsilon}\right)}{\lambda}
\end{array}
\right\}
\geq 1-\varepsilon.
\end{equation*}
\end{lemma}
\noindent {\it Proof of Lemma~\ref{PACBAYES2}.}
Following the proof of Lemma~\ref{PACBAYES} we have:
$$
\mathbb{P} \left\{
\sup_{\rho} \left\{ \lambda \int \left(R(\theta) - r_n (\theta)\right) \rho({\rm
d}\theta)
 - \frac{\lambda^2 \kappa^{2}}{n\left(1-{k}/{n}\right)^2} - \log
\left({2}/{\varepsilon}\right) - \mathcal{K}(\rho,\pi)
 \right\} \geq 0
 \right\} \leq \frac{\varepsilon}{2}.
$$
Now, we use the second inequality of Lemma~\ref{XIAOYIN}, with
$\theta=\overline{\theta}$:
$$
\mathbb{E}\Big(\exp\Big({\lambda \Big(r_n (\overline{\theta}) - R (\overline{\theta})
- \frac{\lambda
\kappa^{2}}{n\left(1-{k}/{n}\right)^2}\Big) }\Big)\Big)
          \leq 1.
$$
But then, we directly apply Markov's inequality to get:
$$
\mathbb{P}\left\{
 r_n(\overline{\theta}) \geq R(\overline{\theta})
       + \frac{\lambda \kappa^2 }{n \left(1-{k}/{n}\right)^{2}}
     + \frac{
\log\left({2}/{\varepsilon}\right)}{\lambda}
\right\} \leq \frac{\varepsilon}{2}.
$$
Here again, a union bound ends the proof.
\hfill $\blacksquare$

\subsection{Proof of Theorems \ref{main_result} and \ref{thmGibbs2}}

In this subsection we prove the general result on the Gibbs predictor.

\noindent {\it Proof of Theorem~\ref{main_result}.}
We apply Lemma~\ref{PACBAYES}. So, with probability at least $1-\varepsilon$
we are on the event given by~\eqref{eqPACBAYES}. From now, we work on that
event. The first inequality of~\eqref{eqPACBAYES}, when applied to $
\hat{\rho}_{\lambda}({\rm d}\theta)$, gives
$$
 \int R(\theta) \hat{\rho}_{\lambda}({\rm d}\theta) \leq  \int r_n (\theta)
\hat{\rho}_{\lambda}({\rm d}\theta)
 + \frac{\lambda \kappa^{2}}{n\left(1-{k}/{n}\right)^2} +\frac{1}{\lambda}
\log \left({2}/{\varepsilon}\right)
+ \frac{1}{\lambda} \mathcal{K}(\hat{\rho}_{\lambda},\pi).
$$
According to Lemma \ref{LEGENDRE} we have:
$$
\int r_n (\theta) \hat{\rho}_{\lambda}({\rm d}\theta)
 + \frac{1}{\lambda} \mathcal{K}(\hat{\rho}_{\lambda},\pi) = \inf_{\rho}
\left(\int r_n (\theta) \rho({\rm d}\theta)
+ \frac{1}{\lambda}\mathcal{K}(\rho,\pi) \right)
$$
so we obtain
\begin{equation} \label{eq1}
 \int R(\theta) \hat{\rho}_{\lambda}({\rm d}\theta)  \leq \inf_{\rho} \left\{
\int r_n (\theta) \rho({\rm d}\theta)
 + \frac{\lambda \kappa^{2}}{n\left(1-{k}/{n}\right)^2} +
\frac{\mathcal{K}(\rho,\pi) + \log\left({2}/{\varepsilon}\right)}{\lambda}
 \right\}.
\end{equation}
We now estimate from above $r(\theta)$ by $R(\theta)$. 
Applying the second inequality of~\eqref{eqPACBAYES} and plugging it into
Inequality~\ref{eq1}
gives
$$
    \int R(\theta) \hat{\rho}_{\lambda}({\rm d}\theta)  \leq \inf_{\rho} \left\{
\int R {\rm d}\rho+\frac{2}{\lambda}\mathcal{K}(\rho,\pi)
 + \frac{2\lambda \kappa^{2}}{n\left(1-{k}/{n}\right)^2} + \frac{2}{\lambda}
\log\left({2}/{\varepsilon}\right) \right\}.
$$
We end the proof by the remark that $\theta\mapsto R(\theta)$ is convex and
so by Jensen's inequality
$ \int R(\theta) \hat{\rho}_{\lambda}({\rm d}\theta) \geq
R\left( \int \theta \hat{\rho}_{\lambda}({\rm d}\theta) \right)
= R(\hat{\theta}_{\lambda}) .$
\hfill $\blacksquare$

\noindent {\it Proof of Theorem~\ref{thmGibbs2}.}
Let us apply Lemma~\ref{PACBAYES} in each model $\Theta_j$, with a fixed
$\lambda_j>0$ and confidence level $\varepsilon_j>0$. We obtain,
for all $j$,
\begin{equation*}
\mathbb{P} \left\{
\begin{array}{l}
\forall \rho \in\mathcal{M}_{+}^{1}(\Theta_j), \\
 \int R {\rm d} \rho \leq \int r_n {\rm d}\rho
       + \frac{\lambda_j \kappa_j^2 }{n \left(1-{k}/{n}\right)^{2}}
     + \frac{\mathcal{K}(\rho,\pi_j) +
\log\left({2}/{\varepsilon_j}\right)}{\lambda_j}
\\
\text{ and }
\\
 \int r_n {\rm d} \rho \leq \int R {\rm d}\rho
       + \frac{\lambda_j \kappa_j^2 }{n \left(1-{k}/{n}\right)^{2}}
     + \frac{\mathcal{K}(\rho,\pi_j) +
\log\left({2}/{\varepsilon_j}\right)}{\lambda_j}
\end{array}
\right\}
\geq 1-\varepsilon_j.
\end{equation*}
We put $\varepsilon_j = p_j \varepsilon$, a union bound gives leads to:
\begin{equation}
\label{refGibbs2}
\mathbb{P} \left\{
\begin{array}{l}
\forall j\in\{1,...,M\},\quad \forall \rho \in\mathcal{M}_{+}^{1}(\Theta_j), \\
 \int R {\rm d} \rho \leq \int r_n {\rm d}\rho
       + \frac{\lambda_j \kappa_j^2 }{n \left(1-{k}/{n}\right)^{2}}
     + \frac{\mathcal{K}(\rho,\pi_j) +
\log\left(\frac{2}{\varepsilon p_j}\right)}{\lambda_j}
\\
\text{ and }
\\
 \int r_n {\rm d} \rho \leq \int R {\rm d}\rho
       + \frac{\lambda_j \kappa_j^2 }{n \left(1-{k}/{n}\right)^{2}}
     + \frac{\mathcal{K}(\rho,\pi_j) +
\log\left(\frac{2}{\varepsilon p_j}\right)}{\lambda_j}
\end{array}
\right\}
\geq 1-\varepsilon.
\end{equation}
From now, we only work on that event of probability at least $1-\varepsilon$.
Remark that
\begin{align*}
 R(\hat{\theta}) & = R(\hat{\theta}_{\lambda_{\hat{j}},\hat{j}}) \\
          & \leq \int R(\theta) \hat{\rho}_{\lambda_{\hat{j}},\hat{j}}({\rm
d}\theta)
              \text{ by Jensen's inequality} \\
          & \leq \int r_n \hat{\rho}_{\lambda_{\hat{j}},\hat{j}}({\rm d}\theta)
       + \frac{\lambda_j \kappa_j^2 }{n \left(1-{k}/{n}\right)^{2}}
     + \frac{\mathcal{K}(\hat{\rho}_{\lambda_{\hat{j}},\hat{j}},\pi_j) +
\log\left(\frac{2}{\varepsilon p_j}\right)}{\lambda_j}
              \\
          & \quad \quad \text{ by~\eqref{refGibbs2}} \\
        & = \inf_{1\leq j\leq M} \left\{
         \int r_n \hat{\rho}_{\lambda_{j},j}({\rm d}\theta)
       + \frac{\lambda_j \kappa_j^2 }{n \left(1-{k}/{n}\right)^{2}}
     + \frac{\mathcal{K}(\hat{\rho}_{\lambda_{j},j},\pi_j) +
\log\left(\frac{2}{\varepsilon p_j}\right)}{\lambda_j}
          \right\} \\
          & \quad \quad \text{ by definition of } \hat{j}\\
        & = \inf_{1\leq j\leq M} \inf_{\rho\in\mathcal{M}_{+}^{1}(\Theta_j)}
           \left\{
         \int r_n \rho ({\rm d}\theta)
       + \frac{\lambda_j \kappa_j^2 }{n \left(1-{k}/{n}\right)^{2}}
     + \frac{\mathcal{K}(\rho,\pi_j) +
\log\left(\frac{2}{\varepsilon p_j}\right)}{\lambda_j}
          \right\} \\
          & \quad \quad \text{ by Lemma~\ref{LEGENDRE}} \\
   & \leq \inf_{1\leq j\leq M} \inf_{\rho\in\mathcal{M}_{+}^{1}(\Theta_j)}
           \left\{
         \int R \rho ({\rm d}\theta)
       + \frac{2 \lambda_j \kappa_j^2 }{n \left(1-{k}/{n}\right)^{2}}
     + 2 \frac{\mathcal{K}(\rho,\pi_j) +
\log\left(\frac{2}{\varepsilon p_j}\right)}{\lambda_j}
          \right\} \\
          & \quad \quad \text{ by~\eqref{refGibbs2} again} \\
     & \leq \inf_{1\leq j \leq M}
               \inf_{\delta >0} \left\{ R(\overline{\theta}_j) + \delta
           + \frac{2 \lambda_j \kappa_j^2 }{n \left(1-{k}/{n}\right)^{2}}
            + 2 \frac{ d_j \log\left({D_j}/{\delta}\right) +
\log\left(\frac{2}{\varepsilon p_j}\right)}{\lambda_j}
 \right\}               \\
          & \quad \quad \text{ by restricting } \rho
               \text{ as in the proof of Cor.~\ref{thmGibbs1}
page~\pageref{thmGibbs1}} \\
   & \leq \inf_{1\leq j \leq M}
    \left\{ R(\overline{\theta}_j)
           + \frac{2 \lambda_j \kappa_j^2 }{n \left(1-{k}/{n}\right)^{2}}
            + 2 \frac{ d_j \log\left(\frac{D_j  e\lambda_j}{d_j}\right) +
\log\left(\frac{2}{\varepsilon p_j}\right)}{\lambda_j}
 \right\}     \\
                & \quad \quad \text{ by taking } \delta = \frac{d_j }{\lambda_j}
\\
    & = \inf_{1\leq j \leq M}
    \left\{ R(\overline{\theta}_j)
           + \inf_{\lambda>0} \left\{ \frac{2 \lambda \kappa_j^2 }{n
\left(1-{k}/{n}\right)^{2}}
            + 2 \frac{ d_j \log\left(\frac{D_j e \lambda}{d_j}\right) +
\log\left(\frac{2}{\varepsilon p_j}\right)}{\lambda}
  \right\} \right\} \\
                & \quad \quad \text{ by definition of } \lambda_j \\
  & \leq
   \inf_{1\leq j \leq M}
    \left\{ R(\overline{\theta}_j)
           + 2\frac{\kappa_j}{1-{k}/{n}} \left\{ \sqrt{\frac{d_j}{n}}
              \log\left(\frac{D_j e^2 }{\kappa_j} \sqrt{\frac{n}{d_j}} \right)
                     + \frac{\log\left(\frac{2}{\varepsilon p_j}\right)}{\sqrt{
n d_j}}
                \right\}  \right\}.
\end{align*}
\hfill $\blacksquare$

\subsection{Proof of Theorems~\ref{thfinite} and \ref{thmERM}}

Let us now prove the results about the ERM.

\noindent {\it Proof of Theorem~\ref{thfinite}.}
We choose $\pi$ as the uniform probability distribution on $\Theta$ and $\lambda>0$.
We apply Lemma~\ref{PACBAYES2}. So we have, with probability at least
$1-\varepsilon$,
$$
\left\{
\begin{array}{l l}
\forall \rho \in\mathcal{M}_{+}^{1}(\Theta'),&
 \int R {\rm d} \rho \leq \int r_n {\rm d}\rho
       + \frac{\lambda \kappa^2 }{n \left(1-{k}/{n}\right)^{2}}
     + \frac{\mathcal{K}(\rho,\pi) +
\log\left({2}/{\varepsilon}\right)}{\lambda}
\\
\text{ and } &
  r_n(\overline{\theta}) \leq R(\overline{\theta})
       + \frac{\lambda \kappa^2 }{n \left(1-{k}/{n}\right)^{2}}
     + \frac{
\log\left({2}/{\varepsilon}\right)}{\lambda}.
\end{array}
\right.
$$
We restrict the $\inf$ in the first inequality to Dirac masses
$\rho\in\{\delta_{\theta},\theta\in\Theta\}$ and we obtain:
$$
\left\{
\begin{array}{l l}
\forall \theta  \in \Theta,&
 R(\theta) \leq  r_n(\theta)
       + \frac{\lambda \kappa^2 }{n \left(1-{k}/{n}\right)^{2}}
     + \frac{
\log\left(\frac{2M}{\varepsilon}\right)}{\lambda}
\\
\text{ and } &
  r_n(\overline{\theta}) \leq R(\overline{\theta})
       + \frac{\lambda \kappa^2 }{n \left(1-{k}/{n}\right)^{2}}
     + \frac{
\log\left({2}/{\varepsilon}\right)}{\lambda}.
\end{array}
\right.
$$
In particular, we apply the first inequality to $\hat{\theta}^{ERM}$.
We remind that $\overline{\theta}$ minimizes $R$ on $\Theta$ and that
$\hat{\theta}^{ERM}$ minimizes $r_n$ on $\Theta$, and so we have
\begin{align*}
R(\hat{\theta}^{ERM})
 & \leq r_n(\hat{\theta}^{ERM})
+ \frac{\lambda \kappa^2 }{n \left(1-{k}/{n}\right)^{2}}
     + \frac{ \log(M)+
\log\left( {2}/{\varepsilon}\right)}{\lambda} \\
 & \leq r_n(\overline{\theta})+ \frac{\lambda \kappa^2 }{n
        \left(1-{k}/{n}\right)^{2}}
     + \frac{\log(M)+
\log\left({2}/{\varepsilon}\right)}{\lambda} \\
 & \leq R(\overline{\theta}) + 
\frac{2 \lambda \kappa^2 }{n
        \left(1-{k}/{n}\right)^{2}}
     + \frac{\log(M)+
2 \log\left({2}/{\varepsilon}\right)}{\lambda} \\
& \leq R(\overline{\theta}) + 
\frac{2 \lambda \kappa^2 }{n
        \left(1-{k}/{n}\right)^{2}}
     + \frac{
2 \log\left({2M}/{\varepsilon}\right)}{\lambda}.
\end{align*}
The result still holds if we choose $\lambda$ as a minimizer of
$$ \frac{2 \lambda \kappa^2 }{n
        \left(1-{k}/{n}\right)^{2}}
     + \frac{
2 \log\left({2M}/{\varepsilon}\right)}{\lambda} .$$
\hfill $\blacksquare$

\noindent {\it Proof of Theorem~\ref{thmERM}.}
We put $\Theta'=\{\theta\in\mathbb{R}^{d}:\|\theta\|_1 \leq D+1\}$.
We choose $\pi$ as the uniform probability distribution on $\Theta'$.
We apply Lemma~\ref{PACBAYES2}. So we have, with probability at least
$1-\varepsilon$,
$$
\left\{
\begin{array}{l l}
\forall \rho \in\mathcal{M}_{+}^{1}(\Theta'),&
 \int R {\rm d} \rho \leq \int r_n {\rm d}\rho
       + \frac{\lambda \kappa^2 }{n \left(1-{k}/{n}\right)^{2}}
     + \frac{\mathcal{K}(\rho,\pi) +
\log\left({2}/{\varepsilon}\right)}{\lambda}
\\
\text{ and } &
  r_n(\overline{\theta}) \leq R(\overline{\theta})
       + \frac{\lambda \kappa^2 }{n \left(1-{k}/{n}\right)^{2}}
     + \frac{
\log\left({2}/{\varepsilon}\right)}{\lambda}.
\end{array}
\right.
$$
So for any $\rho$,
\begin{multline*}
 R(\hat{\theta}^{ERM})  = \int[R(\hat{\theta}^{ERM}) - R(\theta)]\rho({\rm
d}\theta)
         + \int R {\rm d} \rho \\
 \leq \int[R(\hat{\theta}^{ERM}) - R(\theta)]\rho({\rm d}\theta)
 + \int r_n {\rm d}\rho
       + \frac{\lambda \kappa^2 }{n \left(1-{k}/{n}\right)^{2}}
     + \frac{\mathcal{K}(\rho,\pi) +
\log\left({2}/{\varepsilon}\right)}{\lambda} \\
 \leq \int[R(\hat{\theta}^{ERM}) - R(\theta)]\rho({\rm d}\theta)
 + \int [r_n(\theta)-r_n(\hat{\theta}^{ERM})] \rho({\rm d}\theta)
        + r_n(\hat{\theta}^{ERM}) \\
      \shoveright{+ \frac{\lambda \kappa^2 }{n \left(1-{k}/{n}\right)^{2}}
     + \frac{\mathcal{K}(\rho,\pi) +
\log\left({2}/{\varepsilon}\right)}{\lambda}} \\
\leq 2 K \psi \int \|\theta-\hat{\theta}^{ERM}\|_{1} \rho({\rm d}\theta)
        + r_n(\overline{\theta})
      + \frac{\lambda \kappa^2 }{n \left(1-{k}/{n}\right)^{2}}
     + \frac{\mathcal{K}(\rho,\pi) +
\log\left({2}/{\varepsilon}\right)}{\lambda} \\
\leq 2 K \psi \int \|\theta-\hat{\theta}^{ERM}\|_{1} \rho({\rm d}\theta)
        + R(\overline{\theta})
      + \frac{2 \lambda \kappa^2 }{n \left(1-{k}/{n}\right)^{2}}
     + \frac{\mathcal{K}(\rho,\pi) +
2 \log\left({2}/{\varepsilon}\right)}{\lambda}.
\end{multline*}
Now we define, for any $\delta>0$, $ \rho_{\delta}$ by
$$ \frac{ {\rm d} \rho_{\delta} }{{\rm d}\pi}(\theta) =
    \frac{\mathbf{1} \{\|\theta-\hat{\theta}^{ERM}\| < \delta \}}
   { \int_{t\in\Theta'}  \mathbf{1} \{ \|t-\hat{\theta}^{ERM}\|
             < \delta \}  \pi({\rm d}t)     }.
$$
So in particular, we have, for any $\delta>0$,
\begin{multline*}
 R(\hat{\theta}^{ERM})
 \leq 2 K \psi \delta
        + R(\overline{\theta})
 \\
      + \frac{2 \lambda \kappa^2 }{n \left(1-{k}/{n}\right)^{2}}
     + \frac{\log\frac{1}{\int_{t\in\Theta'}  \mathbf{1} \{
\|t-\hat{\theta}^{ERM}\|
             < \delta \}  \pi({\rm d}t)} +
2 \log\left({2}/{\varepsilon}\right)}{\lambda}.
\end{multline*}
But for any $\delta\leq 1$,
$$ -\log \int_{t\in\Theta'}  \mathbf{1} \{ \|t-\hat{\theta}^{ERM}\|
             < \delta \}  \pi({\rm d}t)
       =   d\log\left(\frac{D+1}{\delta}\right).    $$
So we have
$$
 R(\hat{\theta}^{ERM})
 \leq \inf_{\delta\leq 1} \left\{ 2 K \psi \delta
        + R(\overline{\theta})
      + \frac{2 \lambda \kappa^2 }{n \left(1-{k}/{n}\right)^{2}}
     + \frac{d\log \left(\frac{D+1}{\delta}\right) +
2 \log\left({2}/{\varepsilon}\right)}{\lambda}
\right\}.
$$
We optimize this result by taking $\delta=d/(2 \lambda K \psi)$, which is smaller
than $1$ as soon as $t\geq 2 K \psi/d$, we get:
$$
 R(\hat{\theta}^{ERM})
 \leq R(\overline{\theta})
      + \frac{2 \lambda \kappa^2 }{n \left(1-{k}/{n}\right)^{2}}
     + \frac{ d\log\left(\frac{2eK \psi(D+1)t}{d}\right) +
2 \log\left({2}/{\varepsilon}\right)}{\lambda}.
$$
We just choose $\lambda$ as the minimizer of the r.h.s., subject to
$t\geq 2 K \psi/d$, to end the proof.
\hfill $\blacksquare$

\subsection{Some preliminary lemmas for the proof of Theorem \ref{thmfastrates}}

\begin{lemma}
\label{exprisk}
Under the hypothesis of Theorem \ref{thmfastrates}, we have, for any
$\theta\in\Theta$, for any $0\le \lambda\le
(n-k)/(2k KL\mathcal{B} \mathcal{C})$,
\begin{equation*}
 \mathbb{E} \exp\left\{ \lambda  \left[
  \left(1-\frac{8k\mathcal{C} \lambda }{n-k}\right)
 \left(R(\theta)-R(\overline{\theta})\right) - r(\theta) + r(\overline{\theta})
      \right]\right\}
 \leq 1,
\end{equation*}
and
\begin{equation*}
 \mathbb{E} \exp\left\{ \lambda  \left[
  \left(1+\frac{8k\mathcal{C} \lambda}{n-k}\right)
 \left(R(\overline{\theta})-R(\theta)\right) - r(\overline{\theta}) + r(\theta)
      \right]\right\}
 \leq 1.
\end{equation*}
\end{lemma}
\begin{proof}[Lemma \ref{exprisk}]
We apply Lemma \ref{lapmx}
to $N=n-k$, $Z_i=(X_{i+1},\ldots,X_{i+k})$,
\begin{multline*}
f(Z_i)= \frac{1}{n-k}\Bigl[R(\theta)-R(\overline \theta)
\\
         - \ell\left(X_{i+k},f_{\theta}(X_{i+k-1},\dots,X_{i+1})\right) +
\ell\left(X_{i+k},f_{\overline{\theta}}(X_{i+k-1},\dots,X_{i+1})\right)^{2}
\Bigr],
\end{multline*}
and so
$$ S_{N}(f)=[R(\theta)-R(\overline \theta)- r(\theta)+r(\overline
\theta)],$$
and the $Z_i$ are uniformly mixing with coefficients $\phi_{r}^Z
= \phi_{\lfloor r/q \rfloor }$.
Note that
$1+\sum_{r=1}^{n-q}\sqrt{\phi^{Z}_{r}}=1+\sum_{r=1}^{n-q}\sqrt{\phi_{\lfloor r/k
\rfloor}} \leq k \, \mathcal{C}$
by {\bf PhiMix$(\mathcal{C})$}.
For any $\theta$ and $\theta'$ in $\Theta$ let us put
$$ V(\theta,\theta') = \E
\left\{\left[\ell\Bigl(X_{k+1},f_{\theta}(X_{k},...,X_{1})\Bigr)
-\ell\Bigl(X_{k+1},f_{\theta'}(X_{k},...,X_{1})\Bigr)\right]^{2}\right\}. $$
We are going to apply Lemma~\ref{lapmx}.
Remark that $\sigma^2(f) \le V(\theta,\overline\theta) / (n-k)^2$. Also,
\begin{multline*}
 \Bigl| \ell\left(X_{i+k},f_{\theta}(X_{i+k-1},\dots,X_{i+1})\right) -
\ell\left(X_{i+k},f_{\overline{\theta}}(X_{i+k-1},\dots,X_{i+1})\right)
\Bigr|
\\
\leq
K \left| f_{\theta}(X_{i+k-1},\dots,X_{i+1}) -
f_{\overline{\theta}}(X_{i+k-1},\dots,X_{i+1})\right|
\leq
K L \mathcal{B}
\end{multline*}
where we used {\rm LipLoss$(K)$} for the first inequality and {\rm Lip$(L)$} and
{\rm PhiMix$(\mathcal{B},\mathcal{C})$}
for the second inequality.
This implies that
$\|f\|_\infty\le 2 K L \mathcal{B} / (n-k)$, so we can apply Lemma~\ref{lapmx}
for
any $0\le \lambda\le (n-k)/(2k KL\mathcal{B} \mathcal{C}) ]$,
we have
$$
\ln \E
\exp\left[\lambda
\Bigl(R(\theta)-R(\overline{\theta})-r(\theta)+r(\overline{\theta})\Bigr)\right]
\leq  \frac{8k\mathcal{C} V(\theta,\overline{\theta}) \lambda^2}{n-k} .$$
Notice finally that {\rm Margin$(\mathcal{K})$} leads to
$$
V(\theta,\overline{\theta})
= \mathcal{K} \left[R(\theta)-R(\overline{\theta})\right]
$$
This proves
the first inequality of Lemma~\ref{exprisk}. The second inequality is proved
exacly in the same way, but replacing $f$ by $-f$.
\end{proof}

We are now ready to state the following key Lemma.

\begin{lemma}
\label{PACBAYESolivier}
Under the hypothesis of Theorem \ref{thmfastrates}, we have, for any
$0\le \lambda\le (n-k)/(2k KL\mathcal{B} \mathcal{C})$, for any
$0<\varepsilon<1$,
$$
\mathbb{P} \left\{
\begin{array}{l}
\forall \rho \in\mathcal{M}_{+}^{1}(\Theta), \\
 \left(1-\frac{8k\mathcal{C} \lambda }{n-k}\right)
 \left(\int R {\rm d} \rho - R(\overline{\theta}) \right)\leq \int r {\rm d}\rho
      - r(\overline{\theta})
     + \frac{\mathcal{K}(\rho,\pi) +
\log\left({2}/{\varepsilon}\right)}{\lambda}
\\
\text{ and }
\\
 \int r {\rm d} \rho - r(\overline{\theta})
    \leq \left(\int R {\rm d}\rho - R(\overline{\theta})\right)
          \left(1+\frac{8k\mathcal{C} \lambda }{n-k}\right)
     + \frac{\mathcal{K}(\rho,\pi) +
\log\left({2}/{\varepsilon}\right)}{\lambda}
\end{array}
\right\}
\geq 1-\varepsilon.
$$
\end{lemma}
\noindent {\it Proof of Lemma~\ref{PACBAYESolivier}.}
 Let us fix $\varepsilon$, $\lambda$ and $\theta\in\Theta$,
and apply the first inequality of
Lemma~\ref{exprisk}. We have:
\begin{equation*}
 \mathbb{E} \exp\left\{ \lambda  \left[
  \left(1-\frac{8k\mathcal{C} \lambda }{n-k}\right)
 \left(R(\theta)-R(\overline{\theta})\right) - r(\theta) + r(\overline{\theta})
      \right]\right\}
 \leq 1,
\end{equation*}
and we multiply this result by $\varepsilon/2$ and integrate it with
respect to $\pi({\rm d}\theta)$. Fubini's Theorem gives:
\begin{multline*}
 \mathbb{E} \int \exp\Biggl\{ \lambda  \Biggl[
  \left(1-\frac{8k\mathcal{C} \lambda }{n-k}\right)
 \left(R(\theta)-R(\overline{\theta})\right) - r(\theta) +
r(\overline{\theta})
+\log (\epsilon/2 )
      \Biggr]\Biggr\} \pi({\rm d}\theta)
\\
  \leq \frac{\varepsilon}{2}.
\end{multline*}
We apply Lemma~\ref{LEGENDRE} and we get:
\begin{multline*}
 \mathbb{E}  \exp\Biggl\{ \sup_{\rho} \lambda  \Biggl[
  \left(1-\frac{8k\mathcal{C} \lambda }{n-k}\right)
 \left(\int R {\rm d}\rho-R(\overline{\theta})\right) - \int r {\rm d}\rho
 + r(\overline{\theta})
\\
+\log (\epsilon/2 )
 - \mathcal{K}(\rho,\pi)
      \Biggr]\Biggr\}
  \leq \frac{\varepsilon}{2}.
\end{multline*}
As $e^{x}\geq \mathbf{1}_{\mathbb{R}_{+}}(x)$, we have:
\begin{multline*}
\mathbb{P} \Biggl\{
\sup_{\rho} \lambda  \Biggl[
 \left(1-\frac{8k\mathcal{C} \lambda }{n-k}\right)
 \left(\int R {\rm d}\rho-R(\overline{\theta})\right) - \int r {\rm d}\rho
 + r(\overline{\theta})
      \\
+\log (\epsilon/2 )\Biggr] - \mathcal{K}(\rho,\pi)\geq 0
 \Biggr\}
\leq \frac{\varepsilon}{2}.
\end{multline*}
Let us apply the same arguments starting with the second inequality of
Lemma~\ref{exprisk}. We obtain:
\begin{multline*}
\mathbb{P} \Biggl\{
\sup_{\rho} \lambda  \Biggl[
   \left(1+\frac{8k\mathcal{C} \lambda }{n-k}\right)
 \left(R(\overline{\theta})-\int R {\rm d}\rho\right) - r(\overline{\theta})
   + \int r {\rm d}\rho
\\
     +\log (\epsilon/2 ) - \mathcal{K}(\rho,\pi) \Biggr] \geq 0
 \Biggr\}
\leq \frac{\varepsilon}{2}.
\end{multline*}
A union bound ends the proof.
\hfill $\blacksquare$

\subsection{Proof of Theorem \ref{thmfastrates}}

\noindent {\it Proof of Theorem~\ref{thmfastrates}.}
Fix $0\le \lambda = (n-k)/(4k KL\mathcal{B} \mathcal{C})\wedge (n-k)/(16 k
\mathcal{C})
             \le (n-k)/(2k KL\mathcal{B} \mathcal{C})$.
Applying Lemma~\ref{PACBAYESolivier}, we assume from now that the event
of probability at least $1-\varepsilon$ given by this lemma is satisfied. In
particular
we have
$\forall \rho \in\mathcal{M}_{+}^{1}(\Theta)$,
$$
\int R {\rm d} \rho - R(\overline{\theta})
\leq \frac{ \int r {\rm d}\rho
      - r(\overline{\theta})
     + \frac{\mathcal{K}(\rho,\pi) +
\log\left({2}/{\varepsilon}\right)}{\lambda}
} {
  \left(1-\frac{8k\mathcal{C} \lambda }{n-k}\right)
}.
$$
In particular, thanks to Lemma~\ref{LEGENDRE}, we have:
$$
\int R {\rm d} \hat{\rho}_{\lambda} - R(\overline{\theta})
\leq \inf_{\rho\in\mathcal{M}_{+}^{1}(\Theta)} \frac{ \int r {\rm d}\rho
      - r(\overline{\theta})
     + \frac{\mathcal{K}(\rho,\pi) +
\log\left({2}/{\varepsilon}\right)}{\lambda}
} {
  \left(1-\frac{8k\mathcal{C} \lambda }{n-k}\right)
}.
$$
Now, we apply the second inequality of Lemma~\ref{PACBAYESolivier}:
\begin{multline*}
\int R {\rm d} \hat{\rho}_{\lambda} - R(\overline{\theta})
\\
\leq \inf_{\rho\in\mathcal{M}_{+}^{1}(\Theta)} \frac{ 
\left(1+\frac{8k\mathcal{C} \lambda }{n-k}\right) \left[\int R {\rm d}\rho
      - R(\overline{\theta}) \right]
     + 2 \frac{\mathcal{K}(\rho,\pi) +
\log\left({2}/{\varepsilon}\right)}{\lambda}
} {
  \left(1-\frac{8k\mathcal{C} \lambda }{n-k}\right)
}
\\
\leq
\inf_{ j}
\inf_{\rho\in\mathcal{M}_{+}^{1}(\Theta_j)} \frac{ \left(1+\frac{8k\mathcal{C}
\lambda }{n-k}\right) \left[\int R {\rm d}\rho
      - R(\overline{\theta}) \right]
     + 2 \frac{\mathcal{K}(\rho_j,\pi) +
\log\left(\frac{2}{\varepsilon p_j}\right)}{\lambda}
} {
 \left(1-\frac{8k\mathcal{C} \lambda }{n-k}\right)
}
\\
\leq
\inf_{ j}
\inf_{\delta>0} \frac{  \left(1+\frac{8k\mathcal{C} \lambda }{n-k}\right) \left[
R(\overline{\theta}_j) + \delta
      - R(\overline{\theta}) \right]
     + 2 \frac{d_j \log\left(\frac{D_j}{\delta}\right) +
\log\left(\frac{2}{\varepsilon p_j}\right)}{\lambda}
} {
  \left(1-\frac{8k\mathcal{C} \lambda }{n-k}\right)
}
\end{multline*}
by restricting $\rho$ as in the proof of Theorem~\ref{thmGibbs1}.
First, notice that our choice $\lambda\leq  (n-k)/(16 k\mathcal{C} )$ leads to
\begin{align*}
\int R {\rm d} \hat{\rho}_{\lambda} - R(\overline{\theta})
& \leq
2 \inf_{ j}
\inf_{\delta>0} \left\{  \frac{3}{2} \left[ R(\overline{\theta}_j) + \delta
      - R(\overline{\theta}) \right]
     + 2 \frac{d_j \log\left(\frac{D_j}{\delta}\right)+
\log\left(\frac{2}{\varepsilon p_j}\right)}{\lambda}
\right\} \\
& \leq
4 \inf_{ j}
\inf_{\delta>0} \left\{  R(\overline{\theta}_j) + \delta
      - R(\overline{\theta})
     +  \frac{d_j \log\left(\frac{D_j}{\delta}\right)+
\log\left(\frac{2}{\varepsilon p_j}\right)}{\lambda}
\right\}.
\end{align*}
Taking
$\delta = d_j / \lambda$ leads to
$$
\int R {\rm d} \hat{\rho}_{\lambda} - R(\overline{\theta})
 \leq
4 \inf_{ j} \left\{ R(\overline{\theta}_j)
      - R(\overline{\theta}) 
     + \frac{d_j \log\left(\frac{D_j e \lambda}{d_j}\right) +
\log\left(\frac{2}{\varepsilon p_j}\right)}{\lambda}
\right\}.
$$
Finally, we replace the last occurences of $\lambda$ by its value:
\begin{multline*}
\int R {\rm d} \hat{\rho}_{\lambda} - R(\overline{\theta})
 \\
\leq
4 \inf_{ j} \left\{ R(\overline{\theta}_j)
      - R(\overline{\theta}) 
     + \left(16 k \mathcal{C}\vee4k KL\mathcal{B} \mathcal{C}\right)
   \frac{d_j \log\left(\frac{D_j e (n-k)}{16 k \mathcal{C} d_j}\right) +
\log\left(\frac{2}{\varepsilon p_j}\right)}{n-k}
\right\}.
\end{multline*}
Jensen's inequality leads to:
\begin{multline*}
R\left(\hat{\theta}_{\lambda}\right) - R(\overline{\theta})
 \\
\leq
4 \inf_{ j} \left\{ R(\overline{\theta}_j)
      - R(\overline{\theta}) 
     + 4 k \mathcal{C} \left(4 \vee KL\mathcal{B}\right)
\frac{d_j \log\left(\frac{D_j e (n-k)}{16 k \mathcal{C} d_j}\right) +
\log\left(\frac{2}{\varepsilon p_j}\right)}{n-k}
\right\}.
\end{multline*}
\hfill $\blacksquare$

\end{document}